\documentclass[11pt]{amsart}
\usepackage{amsmath,amssymb,amscd,verbatim}
\usepackage{amssymb}
\usepackage{amsthm}
\usepackage{amsmath} \usepackage{latexsym}
\usepackage{color}%

\newcommand{\mc}{\mathcal}
\newcommand{\mb}{\mathbb}
\newcommand{\sub}{\subseteq}
\newcommand{\subneq}{\varsubsetneq}
\newcommand{\nsub}{\nsubseteq}

\newcommand{\lra}{\Leftrightarrow}

\newcommand{\ra}{\Rightarrow}

\newcommand{\sm}{\setminus}
\newcommand{\al}{\alpha}

\DeclareMathOperator{\rad}{rad}
\newcommand{\op}{\operatorname}

\newcommand{\tmax}{t\op{-Max}}
\newcommand{\wmax}{w\op{-Max}}

\newcommand{\starspec}{\ast\op{-Spec}}
\newcommand{\starmax}{\ast\op{-Max}}
\newcommand{\Spec}{\op{Spec}}
\newcommand{\tspec}{t\op{-Spec}}
\newcommand{\wspec}{w\op{-Spec}}

\newcommand{\Na}{R\langle X \rangle}

\newcommand{\tinv}{\op{Inv_t}}

\newtheorem{thm}{Theorem}[section]
\newtheorem{prop}[thm]{Proposition}
\newtheorem{lemma}[thm]{Lemma}
\newtheorem{cor}[thm]{Corollary}
\theoremstyle{definition}

\newtheorem{ex}[thm]{Example}
 \newtheorem{rem}[thm]{Remark}

\begin{document}
    \title{UNIQUE REPRESENTATION DOMAINS, II}

\author{Said El Baghdadi}

\address{Department of Mathematics, Facult\'{e} des Sciences et
Techniques, P.O. Box 523, Beni Mellal, \,\,\,Morocco}

\email{baghdadi@fstbm.ac.ma}

\author{Stefania Gabelli}

    \address{Dipartimento di Matematica, Universit\`{a} degli Studi Roma
Tre, Largo S.  L.  Murialdo, 1, 00146 Roma, Italy}

\email{gabelli@mat.uniroma3.it}

\author{Muhammad Zafrullah}

    \address{57 Colgate Street, Pocatello, ID 83201, USA}

\email{zafrullah@lohar.com}

\maketitle


\section*{Introduction}
Let $R$ be an integral domain. Two elements $x,y\in R$ are said to
have a greatest common divisor (for short a
GCD) if there exists an element $d\in R$ such that $d\mid x,y$ (or $dR\supseteq xR,yR$) and for all $
r\in R$ dividing $x,y$ (or $rR\supseteq xR,yR$) we have $r\mid d$ (or $rR\supseteq dR$).  
A GCD of two elements $x, y$, if it exists, is unique up to a unit and will be denoted by
 GCD($x,y$).
 Following Kaplansky \cite[page 32]{K},
in part, we say that $R$ is a GCD domain if  any two nonzero
elements of $R$  have a GCD. For a more detailed treatment of
greatest common divisors the reader is referred to \cite[page
32]{K}. 

 Of course if $d$
is a GCD of $x,y$ then $x=x_{1}d$ and $y=y_{1}d$ where $\operatorname{GCD}
(x_{1},y_{1})=1$. Here $\operatorname{GCD}(x_{1},y_{1})=1$ signifies
the fact that every common factor of $x_{1},y_{1}$ is a unit (i.e.
if $tR\supseteq x_{1}R,$ $y_{1}R$ then $t$ is a unit). Now a PID $R$
is slightly more refined than a GCD domain in that for every pair of
nonzero ideals $aR,bR$, we have a
unique ideal $dR$ with $aR+bR=dR$ where $d$ is a GCD of $a,b;$ $
a=a_{1}d,b=b_{1}d$ and $a_{1}R+b_{1}R=R$. So $dR$ can be regarded as the GCD of $aR$, $bR$. Note that in our PID
example the principal ideals
 $aR,bR$ are invertible  and that a Pr\"{u}fer domain is a domain in which every nonzero
finitely generated ideal is invertible. If we regard, for every pair
of invertible ideals $A$ and $B$ of a Pr\"{u}fer domain $R$, the
invertible ideal $C=A+B$ as the GCD of $A$ and $B$ we
find that $C\supseteq A,B$. Hence, $R\supseteq
AC^{-1}=A_{1},BC^{-1}=B_{1}$ and so $A=A_{1}C,B=B_{1}C$ where
$A_{1}+B_{1}=R$. Thus, in a Pr\"{u}fer domain each pair of
invertible ideals has the GCD of sorts. Now in GCD domains various
types of unique factorization have been studied, see for instance
Dan Anderson's recent survey \cite{A}. The aim of this note is to
bring to light similar unique
factorizations of those ideals that behave like invertible ideals of Pr\"{u}fer domains.

It was shown in \cite{ZURD}
that in a GCD domain $R$ a principal ideal $xR$ has finitely many minimal
primes if and only if $xR$ can be uniquely expressed as a product
$xR=(x_{1}R)(x_{2}R)\dots (x_{n}R)$ where each of $x_{i}R$ \ has a
unique minimal prime and $x_{i}$ are coprime in pairs. A GCD domain
in which every proper principal ideal has finitely many minimal
primes was called a unique representation domain (URD) and a nonzero
principal ideal with a unique minimal prime ideal was called a
packet in \cite{ZURD}. Following the terminology of \cite{ZURD} an
invertible integral ideal $I$ in a Pr\"{u}fer domain will be called
a packet if $I$ has a unique minimal prime and a Pr\"ufer domain $R$ a URD if every
proper invertible ideal of $R$ is uniquely expressible as a product of mutually comaximal packets.
One aim of this note is
to show that if $X
$ is an invertible ideal in a Pr\"{u}fer domain such that the set $\operatorname{Min}
(X)$ of minimal primes of $X$ is finite then $X$ is expressible (uniquely)
as a product of mutually comaximal packets. This may be used to describe
the factorization of invertible ideals of Pr\"{u}fer domains in which every
invertible ideal has at most a finite number of minimal primes. This class
includes the Pr\"{u}fer domains of finite character and the so called
generalized Dedekind domains. Our immediate plan is to give a more general definition of URD's,
start with a notion that
generalizes both Pr\"{u}fer and GCD domains, prove analogues of results on
URD's for this generalization and show that results on GCD and Pr\"{u}fer
domains follow from this. The domains that generalize both GCD and
 Pr\"{u}fer domains are called
Pr\"{u}fer $v$-multiplication domains (PVMD's) by
some and $t$-Pr\"{u}fer or pseudo Pr\"ufer by others. Once that is done we
shall look into generalizations of this kind of unique factorization of
ideals. This will make the presentation a bit repetitive but in the presence
of motivation the reader will find the paper more readable.

Before we indicate our plan it seems pertinent to provide a
working introduction to the notions that we shall use in this paper.

Let $R$ be an integral domain with quotient field $K$ and $F(R)$ the
set of nonzero fractional ideals of $R$.  A star operation $\ast $ on $R$ is a map
$F(R)\rightarrow F(R)$, $I\mapsto I^{\ast }$, such that the following
conditions hold for each $0\not=a\in K$ and for each $I,J\in F(R)$:

\begin{itemize}
\item[(i)] $R^\ast = R$ and $(aI)^\ast = aI^\ast$;

\item[(ii)] $I \subseteq I^\ast$, and $I \subseteq J \Rightarrow I^\ast
\subseteq J^\ast$;

\item[(iii)] $I^{\ast\ast} = I^\ast$.
\end{itemize}

For standard material about star
operations, see Sections 32 and 34 of \cite{Gi}.
For our purposes we note the following.

Given two ideals $I,J\in F(R)$ we have $(IJ)^{\ast }=(I^{\ast }J)^{\ast
}=(I^{\ast }J^{\ast })^{\ast }$ ($\ast $-multiplication) and we have $%
(I+J)^{\ast }=(I^{\ast }+J)^{\ast }=(I^{\ast }+J^{\ast })^{\ast }$ ($\ast $%
-addition).

A nonzero fractional ideal $I$ is a $\ast $-ideal if $I=I^{\ast}$ and it is
$\ast$-finite (or of finite type) if $I^\ast=J^{\ast}$ for some finitely
generated ideal $J\in F(R)$. A star operation $\ast$ is of finite type if $
I^{\ast }=\bigcup \{J^{\ast }\,:\,J\subseteq I$ and $J$ is fini\-te\-ly
ge\-ne\-ra\-ted\}, for each $I\in F(R)$. To each star operation $\ast $, we
can associate a star operation of finite type $\ast _{f}$, defined by $
I^{\ast _{f}}=\bigcup \{J^{\ast }\,:\,J\subseteq I$ and $J$ is fini\-te\-ly
ge\-ne\-ra\-ted\}, for each $I\in F(R)$. If $I$ is a finitely generated
ideal then $I^{\ast }=I^{\ast _{f}}$.

 Several star operations can be defined on $R$. The trivial example of a
star operation is the identity operation, called the $d$-operation, $I_{d}=I$
for each $I\in F(R)$. Two nontrivial star operations which have been
intensively studied in the literature are the $v$-operation and the $t$
-operation. Recall that the $v$-closure of an ideal $I\in F(R)$ is $
I_{v}=(I^{-1})^{-1}$, where for any $J\in F(R)$ we set $J^{-1}=(R\colon J)=\{x\in K\,:\, xJ
\subseteq R\}$. A $v$-ideal is also called a divisorial ideal. The $t$
-operation is the star operation of finite type associated to $v$. Thus $
I=I_{t}$ if and only if, for every finite set $x_{1},\dots, x_{n}\in I$ we
have $(x_{1},\dots, x_{n})_{t}\subseteq I$. If $\{R_{\alpha }\}$ is a family of overrings of $R$ such that $R\sub R_\al\sub K$ and $R=\cap R_{\alpha }$,
then, for all $I\in F(R)$, the association $I\mapsto I^{\ast }=\cap
IR_{\alpha }$ is a star operation ``induced" by $\{R_{\alpha }\}$.

If $\ast _{1}$ and $\ast _{2}$
are two star operations defined on $R$ we say that $\ast _{2}$ is coarser
than $\ast _{1}$ (notation $\ast _{1}\leq \ast _{2}$) if, for all $I\in F(R)$
we have $I^{\ast _{1}}\subseteq I^{\ast _{2}}$. If $\ast _{1}\leq \ast _{2}$
then for each $I\in F(R)$ we have $(I^{\ast _{1}})^{\ast _{2}}=(I^{\ast
_{2}})^{\ast _{1}}=I^{\ast _{2}}$. The $v$-operation is the coarsest of all
star operations on $R$ and the $t$-operation is coarsest among 
star operations of finite type.

A $\ast $-prime is a prime ideal which is also a $\ast$-ideal and a $\ast $-maximal ideal is a
$\ast$-ideal maximal in the set of proper integral $\ast $-ideals of $R$. We denote by
$\starspec(R)$ (respectively, $\starmax(R)$) the set of $\ast $-prime
(respectively, $\ast $-maximal) ideals of $R$. If $\ast $ is a star
operation of finite type, by Zorn's lemma each $\ast $-ideal is contained in
a $\ast $-maximal ideal, which is prime. In this case, $R=\bigcap _{M\in \starmax(R)}R_{M}$. We say that $R$
has $\ast $-finite character if each nonzero element of $R$ is contained in
at most finitely many $\ast $-maximal ideals.

The star operation induced by $\{R_{M}\}_{M\in \tmax(R)}$, denoted by $w$, was introduced by Wang and McCasland in  \cite{WM}. The star operation induced by 
$\{R_{M}\}_{M\in \starmax(R)}$, denoted by $\ast _{w}$, was studied \ by Anderson
and Cook in \cite{AC}, where it was shown that $\ast _{w}\leq \ast$ and that 
$\ast _{w}$ is of finite type. 

When $\ast $ is of finite type, a minimal prime of a $\ast$-ideal is
a $\ast$-prime. In particular, any minimal prime over a nonzero
principal ideal is a $\ast$-prime for any star operation $\ast$ of finite
type  \cite{HH}.  An height-one prime ideal, being minimal over a
principal ideal, is a $t$-ideal. We say that $R$ has $t$-dimension
one if each $t$-prime ideal has height one.

For any star operation $\ast$, the set of fractional $\ast$-ideals is a
semigroup under the $\ast$-multiplication $(I, J)\mapsto (IJ)^\ast$, with
unity $R$. An ideal $I\in F(R)$ is called $\ast$-invertible if $I^\ast$ is
invertible with respect to the $\ast$-multiplication, i.e., 
$(II^{-1})^\ast=R$. If $\ast$ is a star operation of finite type, then a
$\ast$-invertible ideal is $\ast$-finite. 
For concepts related to
star invertibility and the $w$-operation the readers may consult \cite{Zputting} and if
need arises references there. For our purposes we note that for a
finite type star operation $\ast,$ $R$ is a Pr\"ufer $\ast$-multiplication domain (for short a P$\ast $MD) if every nonzero finitely generated ideal of $R$ is $\ast $-invertible. A
P$t$MD is denoted by PVMD for historical reasons.

Finally two $\ast $-ideals $I,J$ are $\ast $-comaximal if $(I+J)^{\ast }=R$.
If $\ast $ is of finite type, $I,J$ are $\ast $-comaximal if and only if $I+J $
is not contained in any $\ast $-maximal ideal. We note that the
following statement can be easily proved for $\ast $ of finite character.
 If $\{I_{\al}\}$ is a finite family of pairwise $\ast $-comaximal ideals, then 
$(\bigcap I_{\al})^{\ast }=(\prod I_{\al})^{\ast }$. In fact we have
$(\bigcap I_{\al})^{\ast
_{w}}=\bigcap_{M\in \starmax(R)}(\bigcap I_{\al})R_{M}=\bigcap_{M\in \starmax
(R)}(\prod I_{\al})R_{M}=(\prod I_{\al})^{\ast _{w}}$,
where the  second equality holds because at most one $I_{\al}$ survives in $R_{M}$. 
Now the result follows from the
fact that $\ast $ is coarser than $\ast _{w}$.

For a finite type star operation $\ast $, we call a $\ast$-invertible $\ast$-ideal
$I\sub R$ a $\ast$-packet if $I$ has a unique minimal prime ideal and
 we call $R$ a $\ast$-URD if every proper
 $\ast$-invertible $\ast$-ideal of $R$ can be uniquely expressed  as a
$\ast$-product of finitely many mutually $\ast$-comaximal $\ast$-packets. When $\ast=t$, a $\ast$-packet will be called a packet and the $\ast$-URD
will be called a URD.

In the first section we show that if $R$ is a PVMD and $I$ a
$t$-invertible $t$-ideal then $I$ has finitely many minimal primes
if and only if $I$ is expressible as a $t$-product of finitely many
packets, if and only if $I$ is expressible as a (finite) product of
mutually $t$-comaximal packets. We also show that a PVMD $R$ is a
URD if and only if every proper principal ideal of $R$ has finitely
many minimal primes. The results in this section are a translation
of results proved in \cite{ZURD}, establishing yet again the
folklore observation that most of the multiplicative results proved
for GCD domains can be proved for PVMD's.

In Section 2,
for a finite type star operation $\ast $ we prove that a domain $R$ is a $\ast $-URD if and only if $\ast
$-Spec($R$) is treed and every proper principal ideal of $R$ has
finitely many minimal prime ideals. Thus Section 2 serves to isolate
the main conditions that allowed unique representation in PVMD's. On
the other hand the generality of $\ast $ lets us prove, in one go,
results about URD's and about $d$-URD's which, incidentally, were
the main topic in \cite{BH}.

Finally, in Section 3, we study methods of making new URD's from
given ones.
We show that if $R$ is a URD and $S$ is a multiplicative set of $R$ then $%
R_{S}$ is a URD. Our main result here is the characterization of
when a polynomial ring is a URD. We close the section by studying
under what conditions we can construct URD's using the $D+XD_{S}[X]$
construction.

\section{From GCD-domains to PVMD's}\label{one}

Let $R$ be a  PVMD. If $X, Y\sub R$ are
$t$-invertible $t$-ideals and $\Delta =(X+Y)_{t}$, then $\Delta $ is
a $t$-invertible $t$-ideal, $\Delta \supseteq X,Y$ and
$X_{1}=(X\Delta ^{-1})_{t}$, $Y_{1}=(Y\Delta ^{-1})_{t}\subseteq R$
are $t$-comaximal $t$-invertible $t$-ideals; thus, as in the case of
Pr\"{u}fer domains, we can say that $\Delta $ is the GCD of $X,Y$.
So we can carry out the plans for PVMD's that we made, in the
introduction, with the class of
 Pr\"{u}fer domains as our model. We intend to prove in this section the
following two basic theorems involving packets. 

\begin{thm}\label{Theorem A} Let $R$ be a PVMD and $X\sub R$  a $t$-invertible $t$-ideal. If $X$ is a $t$-product of a finite number of packets then $X$ can be uniquely
expressed as a $t$-product of a finite number of mutually
$t$-comaximal packets.
\end{thm}

\begin{thm}\label{Theorem B}  Let $R$ be a PVMD and $X\subneq R$  a $t$-invertible $t$-ideal. Then $X$ is a $t$-product of finitely many packets if and only if $X$ has finitely many minimal primes.
\end{thm}

In a GCD (respectively, Pr\"ufer) domain a $t$-invertible $t$-ideal
is principal (respectively, invertible) and the $t$-product is just
the ordinary product. Indeed it may be noted that a domain $R$ can
have all $t$-invertible $t$-ideals as $t$-products  of packets
without $R$ being a PVMD as we shall point out in the course of our
study.

From this point on we shall use the $v$ and $t$ operations freely.
The results of this section are based on the following observations
that, if reference or proof is not indicated, can be traced back to
Mott and Zafrullah \cite{MZ} or to Griffin \cite{Gr}.

\begin{enumerate}
\item[(I)] A nonzero prime ideal $P$ of a PVMD $R$ is essential (i. e., $R_P$ is a valuation domain) if and only if $P$ is a $t$-ideal.

\item[(II)] The set of prime $t$-ideals of a PVMD $R$ is a tree under inclusion, that is any $t$-maximal ideal of $R$ cannot contain two incomparable $t$-primes.

\item[(III)] If $X$ is a $t$-invertible $t$-ideal and $P$ is a prime
$t$-ideal of a PVMD $R$ then $XR_{P}$ is principal. ($XR_{P}$ is an
invertible ideal \cite{BZ} and in a valuation domain
invertible ideals are principal.)

\item[(IV)] If $X=(X_{1}X_{2})_{t}$ then in view of the definition of $t$
-multiplication we shall assume that both $X_{i}$ are $t$-ideals.

\end{enumerate}

 Let us start with some lemmas on packets in a PVMD.

\begin{lemma}\label{lemma1} Let $R$ be a PVMD and let $X, Y\subneq
R$ be $t$-invertible $t$-ideals such that $\Delta=(X+Y)_{t} \neq R$.
Then the following conditions hold:
\begin{itemize}

\item[(1)] If $P$ is a prime ideal minimal over $\Delta $ then $P$
is minimal over $X$ or $Y$.

\item[(2)] If $P$ is a prime ideal minimal over $X$ such that $Y\subseteq P$
then $P$ is minimal over $\Delta$.

\item[(3)] If $X$ and $Y$ are packets then so are $\Delta $ and
$(XY)_{t}$.
\end{itemize}
\end{lemma}
\begin{proof} (1). Note that being minimal over $\Delta,$ $P$ is a
$t$-ideal. Clearly $X,Y\subseteq P$. So $P$ contains a prime ideal
$P_{1}$ that is
minimal over $X$ and $P_{2}$ that is minimal over $Y$. By (II), $P_{1}$ and 
$P_{2}$ are comparable being minimal primes themselves. Say $P_{2}$ 
$\supseteq P_{1}$. Then both $X,Y\subseteq P_{2}$. Thus $\Delta
=(X+Y)_{t}\subseteq P_{2}$. But $P$ is minimal over $\Delta$. So,
$P=P_{2}$ and hence $P$ is minimal over $Y$. Similarly if $P_{1}$
had contained $P_{2}$ we would have concluded that $P$ was minimal
over $X$.

(2). Since $P\supseteq X,Y$ we conclude that $P\supseteq
(X+Y)_{t}=\Delta$. So, if $P$ is not minimal over $\Delta $ then
$P$ contains a prime $Q$ that is minimal over $\Delta$. Again as
$Q\supseteq \Delta =(X+Y)_{t},$  $Q$ contains a prime minimal
over $X$. But $Q$ is contained in $P$.

(3). Let $P_{1}\not=P_{2}$ be two minimal primes over $\triangle $.
By (1), $P_{i}$ is minimal over $X$ or over $Y$. Assume that $P_{1}$
is (the unique) minimal prime over $X$, then $P_{2}$ is the minimal
prime over $Y$. But $Y\subseteq P_{1}$, then $P_{2}\subneq P_{1}$,
which contradicts the minimality of $P_{1}$ over $\Delta $. A
similar argument applies if $P_{2}$ is the minimal prime over $X$.
For $(XY)_{t}$, assume that $P_{1}$ and $P_{2}$ are minimal primes
over $(XY)_{t}$, then $P_{i}$ is minimal over $X$ or over $Y$. We
can assume that $P_{1}$ is the minimal prime over $X$. Then $P_{2}$
is the minimal prime over $Y$. On the other hand, since $\triangle
=(X,Y)_{t}\not=R$ is a packet, the unique minimal prime over
$\triangle $ is $P_{1}$ or $P_{2}$ (by (1)), this forces
$P_{1}\subseteq P_{2}$
or $P_{2}\subseteq P_{1}$. In either case minimality of both forces $%
P_{1}=P_{2}$.
\end{proof}

\begin{lemma}\label{lemma2} Let $R$ be a PVMD and let $P$ be a prime ideal minimal over a $t$-invertible $t$-ideal $X$. Then for all $t$-invertible $t$-ideals $Y$ contained
in $P$ there exists a natural number $n$ such that
$XR_{P}\varsupsetneq Y^{n}R_{P}$.
\end{lemma}

\begin{proof} Note that $P$ is a prime $t$-ideal in a PVMD and so
$R_{P}$ is a valuation domain and $XR_{P},$ $YR_{P}$ are principal.
Now suppose on the contrary that $XR_{P}\subseteq Y^{n}R_{P}$ for
all natural numbers $n$. Then $ XR_{P}\subseteq \bigcap_n
Y^{n}R_{P}$ and $\bigcap_n Y^{n}R_{P}$ is a prime ideal properly
contained in $PR_{P}$ and this contradicts the minimality of $P$
over $X$. Now since in a valuation domain every pair of ideals is
comparable we have the result.
\end{proof}

\begin{lemma}\label{lemma3} Let $R$ be a PVMD and let $X_{1}, X_{2}\subneq R$ be $t$-invertible $t$-ideals. Let
 $W_{i}$ be the set of $t$-maximal ideals containing 
 $X_{i}$, for $i=1, 2$. If $W_{2}\subseteq W_{1}$ and
 $X_{2}$ is not contained in any minimal prime of $X_{1}$,
then $X_{1}\subseteq (X_{2}^{n})_{t}$, for all natural numbers $n$.

Consequently if $X, Y$ are two packets with $
(X+Y)_{t}\neq R$ such that $\rad(X)\varsupsetneq \rad(Y)$ then $(X^{n})_{t}\supseteq Y$, for all natural numbers $n$.
\end{lemma}

\begin{proof} Clearly for every $t$-maximal ideal $P$ in $W_{1}$ (and hence
for every $t$-maximal ideal $P)$ we have $X_{1}R_{P}\subseteq
X_{2}^{n}R_{P}$
for each $n$ because $X_{2}$ is not contained in the minimal prime of $
X_{1}R_{P}$. Thus for each $n,$ $\bigcap_P X_{1}R_{P}\subseteq \bigcap_P
X_{2}^{n}R_{P} $ or $(X_{1})_{w}\subseteq (X_{2}^{n})_{w}$. Since
in a PVMD for every nonzero ideal $I$ we have $I_{w}=I_{t}$
\cite[Theorem 3.5]{Kang}, and since $X_{1}$ is a $t$-ideal, we
conclude that for each $n$ we have $X_{1}\subseteq
(X_{2}^{n})_{t}$.

The proof of the consequently part
follows from the fact that every $t$-maximal ideal that contains $X$ contains $%
\rad(X)$ and hence $Y$. On the other hand, a $t$-maximal ideal
that contains $Y$ may or may not contain $X$ and $X$ has only one
minimal prime ideal which properly contains the minimal prime
ideal of $Y$.
\end{proof}

\begin{lemma}\label{lemma4} Let $R$ be a PVMD and let $X\subneq R$ be a $t$-invertible $t$-ideal.
Then the following conditions are equivalent:
\begin{itemize}
\item[(i)] $X$ is a packet;

\item[(ii)] for every $t$-factorization $X=(X_{1}X_{2})_{t},$
$X_{1}\supseteq (X_{2}^{2})_{t}$ or $X_{2}\supseteq
(X_{1}^{2})_{t}$.
\end{itemize}
\end{lemma}
\begin{proof} (i) $\Rightarrow $ (ii). Let $Q$ be the unique minimal
prime of $X$ and let $X=(X_{1}X_{2})_{t}$. Indeed we can assume
$X_{i}\in \tinv(R)$. If either of $X_{i}=R$ (ii) holds, so let us
assume that both $X_{i}$ are proper. If either of the $X_{i}$ is not
contained in $Q$ the result follows from Lemma \ref{lemma3}. If both
are contained in $Q$ then, because $Q$ is a $t$-ideal, $\Delta=
(X_{1}+X_{2})_{t} \subseteq Q$. As described above we can write $
X_{1}=(Y_{1}\Delta )_{t}$ and $X_{2}=(Y_{2}\Delta )_{t}$ where
$Y_{i}\in \tinv(R)$ such that $(Y_{1}+Y_{2})_{t}=R$. Now $Y_{1},
Y_{2}$ being $t$-comaximal cannot share a prime $t$-ideal. If
$Y_{1}\nsubseteq Q$ then by Lemma \ref{lemma3}, $ \Delta \subseteq
(Y_{1}^{n})_{t}$ for all $n$ and so $\Delta
^{2}\subseteq (Y_{1}^{n})_{t}\Delta \subseteq X_{1}$,  for any $n$. Combining with
$ X_{2}^{2}\subseteq \Delta ^{2}$ and applying the $t$-operation we
conclude that $(X_{2}^{2})_{t}\subseteq X_{1}$. Same arguments apply
if $ Y_{2}\nsubseteq Q$.

(ii) $\Rightarrow$ (i). Suppose that $X$ satisfies (ii) and let, by
way of contradiction, $P$ and $Q$ be two distinct prime ideals
minimal over $X$. Now $P$ and $Q$ are prime $t$-ideals and so both
$R_{P}$ and $R_{Q}$ are valuation domains. Let $y\in P\backslash Q$.
Then $\Delta =(X+yR)_{t} \subseteq P$ and $\Delta \nsubseteq
Q$. By Lemma \ref{lemma2} there exists $n$ such that $\Delta
^{n}R_{P}\subsetneq XR_{P}$ and consequently $(\Delta
^{n})_{t}\nsupseteq X$. Consider $K=(X+\Delta ^{n})_{t}$. We can
write $ X=(AK)_{t}$ and $(\Delta ^{n})_{t}=(BK)_{t}$ where $A,B\in
\tinv(R)$ and $ (A+B)_{t}=R$; thus $A$ and $B$ cannot share a
minimal $t$-prime ideal. Since  $\Delta
^{n}\nsubseteq Q$ and since $K\supseteq \Delta ^{n}$ we conclude
that $K\nsubseteq Q$. But since $X=(AK)_{t}\subseteq Q$ we conclude
that $A\subseteq Q$. Next, we claim that $A\nsubseteq P$. For if
$A\subseteq P$ then since $(A+B)_{t}=R$, $B\nsubseteq P$ and
$BKR_{P}=\Delta ^{n}R_{P}\supseteq AKR_{P}=XR_{P}$. This contradicts
the fact that $\Delta ^{n}R_{P}\subsetneq XR_{P}$. Thus we have
$X=(AK)_{t}$ with

(a) $K\subseteq P$ and $K\nsubseteq Q$ and

(b) $A\subseteq Q$ and $A\nsubseteq P$

Now by (ii) $X=(AK)_{t}$ where $A\supseteq (K^{2})_{t}$ or
$K\supseteq
(A^{2})_{t}$. If $A\supseteq (K^{2})_{t}$ then as $A\subseteq Q$ we have $%
K\subseteq Q$ which contradicts (a). Next if $K\supseteq
(A^{2})_{t}$ then since $K\subseteq P$ we have $A\subseteq P$ and
this contradicts (b). Of course these contradictions arise from the
assumption that $X$ satisfying (ii) can have more than one minimal
prime.
\end{proof}

\begin{lemma}\label{lemma5} Let $A$ and $B$ be two ideals of a domain $R$. Then the following statements hold:
\begin{itemize}
\item[(1)] If
$(A+B)_{t}=R$ and  $A_{t}\supseteq (XB)_{t}$ for some nonzero ideal $X$, then $
A_{t}\supseteq X_{t}$.

\item[(2)] If $B=B_{1}B_{2}\ $then $(A+B)_{t}=R$ if and only if $
(A+B_{i})_{t}=R$ for $i=1, 2$.

\item[(3)] If $R$ is a PVMD and $X$ is a packet such that $X=(X_{1}X_{2})_{t},$ with $
(X_{1}+X_{2})_{t}=R$, then $(X_{1})_{t}=R$ or $(X_{2})_{t}=R$.

\end{itemize}
\end{lemma}
\begin{proof} (1). Note that $A_{t}\supseteq A_{t}+(XA)_{t}\supseteq
(XB)_{t}+(XA)_{t}$. Applying the $t$-operation $A_{t}\supseteq
((XB)_{t}+(XA)_{t})_{t}=(XB+XA)_{t}=(X(A+B))_{t}=X_{t}$.

(2). Suppose that $(A+B_{1}B_{2})_{t}=R$. Then $R=(A+B_{1}B_{2})_{t}
\subseteq (A+B_{i})_{t}\subseteq R$ for $i=1,2$. Conversely \ if $
(A+B_{i})_{t}=R$ for $i=1,2$ then $%
(A+B_{1}B_{2})_{t}=(A+AB_{1}+B_{1}B_{2})_{t}=(A+B_{1}(A+B_{2}))_{t}=(A+B_{1})_{t}.
$

(3). We note that $(X_{1})_{t}\supseteq (X_{2}^{2})_{t}$ or $%
(X_{2})_{t}\supseteq (X_{1}^{2})_{t}$ (Lemma \ref{lemma4}). If
$(X_{1})_{t}\supseteq
(X_{2}^{2})_{t}$ then $(X_{1})_{t}=(X_{1}+X_{2}^{2})_{t}=R$ by (2) because $%
(X_{1}+X_{2})_{t}=R$.
\end{proof}

\begin{lemma}\label{lemma6} Let $R$ be a PVMD and let $X,Y_{1},Y_{2}\sub R$ be $t$-invertible $t$-ideals.
If $X\supseteq (Y_{1}Y_{2})_{t}$ then $X=(X_{1}X_{2})_{t}$ where $%
X_{i}\supseteq Y_{i}$.
\end{lemma}
\begin{proof} Let $X_{1}=(X+Y_{1})_{t}$. Then $X=(AX_{1})_{t}$ and $%
Y_{1}=(BX_{1})_{t}$ where $(A+B)_{t}=R$. Now
$(AX_{1})_{t}\supseteq (BX_{1}Y_{2})_{t}$. Multiplying both sides
by $X_{1}^{-1}$ and applying $t$ we get
$A=(AX_{1}X_{1}^{-1})_{t}\supseteq
(BX_{1}X_{1}^{-1}Y_{2})_{t}=(BY_{2})_{t}$. Now $A\supseteq
(BY_{2})_{t}$ and $(A+B)_{t}=R$ and so by Lemma \ref{lemma5}
$A\supseteq Y_{2}$. Setting $A=X_{2}$ gives the result.
\end{proof}

The above result is new in principle but not in spirit. Let us
recall that a
domain $R$ is a pre-Schreier domain if for $a, b_{1}, b_{2}\in R\sm \{0\}$, $
a\mid b_{1}b_{2}$ implies that $a=a_{1}a_{2}$ where $a_{i}\mid
b_{i}$ for each $i=1, 2$ \cite{ZPresch}. Pre-Schreier generalizes
the notion of a Schreier domain introduced by Cohn \cite{C} as a
generalization of a GCD domain. In simple terms, a Schreier domain
is an integrally closed pre-Schreier domain. In \cite{DM}, a
domain $R$ was called quasi-Schreier if, for invertible ideals $X,Y_{1},Y_{2}\sub R$,
 $X\supseteq Y_{1}Y_{2}$ implies that $X=X_{1}X_{2}$ where $
X_{i}\supseteq Y_{i}$. It was shown in \cite{DM} that a Pr\"ufer
domain is quasi-Schreier, a result which can also be derived from
Lemma \ref{lemma6} by noting that a Pr\"ufer domain is a PVMD in
which every nonzero ideal is a $t$-ideal. One may tend to define a
domain $R$ as $t$-quasi-Schreier if,  for $t$-invertible $t$-ideals
$X,Y_{1},Y_{2}\sub R$,
 $X\supseteq (Y_{1}Y_{2})_t$ implies that $X=(X_{1}X_{2})_t$ where $
X_{i}\supseteq Y_{i}$, but this
does not seem to be a place for studying this concept.

\bigskip
{\it Proof of Theorem \ref{Theorem A}.} Let $R$ be a PVMD and let $X\subneq R$ such that
 $X=(X_{1}X_{2}\dots  X_{n})_{t}$, where $X_{i}$ are packets. By (3) of
  Lemma \ref{lemma1} the product of two non $t$-comaximal
packets is again a packet. So with a repeated use of  (3) of Lemma
\ref{lemma1} we can reduce the above expression to
$X=(Y_{1}Y_{2}\dots Y_{r})_{t}$ where $(Y_{i}+Y_{j})_{t}=R$ for $i$
$\neq j$. Now let $X^{\prime }$ be a packet such that $X^{\prime
}\supseteq X=(Y_{1}Y_{2}\dots Y_{r})_{t}$. Then
by Lemma \ref{lemma6}, $X^{\prime }=(UV)_{t}$ where $U_{t}\supseteq Y_{1}$ and 
$V_{t}\supseteq (Y_{2}\dots Y_{r})_{t}$. This gives
$U_{t}+V_{t}\supseteq
Y_{1}+(Y_{2}\dots Y_{r})_{t}$. Applying the $t$-operation, using the fact that $(Y_{i}+Y_{j})_{t}=R$ for $i\neq j$ and using Lemma \ref{lemma5} we
conclude that one of $U_{t},$ $V_{t}$ is $R$. If $V_{t}=R$ then
$X^{\prime }$ $\supseteq Y_{1}$ and $X^{\prime }$ does not contain
any of the  other $Y_{i}$. If $U_{t}=R$ then $X^{\prime }\supseteq
(Y_{2}\dots Y_{r})_{t}$ and using the above procedure we can in the
end conclude that $X^{\prime }$ contains precisely one of the
$Y_{i}$. So if $X$ has another expression $X=(B_{1}B_{2}\dots
B_{s})_{t}$ where $B_{i}$ are packets such that
$(B_{i}+B_{j})_{t}=R$ for $i\neq j$ then by the above procedure each
of $X_{i}$ contains precisely one of the $B_{j}$ and $B_{j}$
contains precisely one of the $X_{k}$. But then $i=k$. For if not then $X_{i}\supseteq X_{k}$ so $X_{i}=(X_{i}+X_{k})_{t}=R$ which is
impossible. Thus $r=s$ and each $X_{i}$ is equal to precisely one
of the $B_{j}$.

\bigskip
 {\it Proof of Theorem \ref{Theorem B}.} Let $R$ be a PVMD and
let $X\in \tinv(R)$ such that $X$ is a finite product of packets.
Then by Theorem \ref{Theorem A}, $X$ can be uniquely expressed as a
finite product of mutually $t$-comaximal packets. That $X$ has only
finitely many minimal primes follows from the fact that any minimal
prime over $X$ is minimal over some packet in its factorization.

 We prove the converse by induction
on the number of minimal primes of $X\in \tinv(R)$.  If $X$ has a
single minimal prime ideal then $X$ is a packet and we have nothing
to prove. Suppose that if $X$ has $k-1$ minimal primes then $X$ is
expressible as a product of $k-1$ mutually $t$-comaximal packets.
Now suppose that $P_{1},P_{2},\dots,P_{k}$ are all the minimal
primes of $X$. Choose $y\in P_{1}\backslash (P_{2}\cup P_{3}\cup
\dots \cup P_{k})$. By Lemma \ref{lemma2}, there exists a natural
number $n$ such that $XR_{P_{1}}\varsupsetneq y^{n}R_{P_{1}}$. Let
$\Delta =(y^{n}R+X)_{t}$ with $X=(\Delta X_{1})_{t}$ and
$y^{n}R=(\Delta Y_{1})_{t}$ \ where $(X_{1}+Y_{1})_{t}=R$ and
$X_{1},Y_{1}\in \tinv(R)$. Noting that in $R_{P_{1}},$
$A_{t}R_{P_{1}}=AR_{P_{1}}$ is principal if $A$ is $t$ -invertible,
we convert $XR_{P_{1}}\varsupsetneq y^{n}R_{P_{1}}$ into $ \Delta
X_{1}R_{P_{1}}\varsupsetneq \Delta Y_{1}R_{P_{1}}$ which reduces to
$ X_{1}R_{P_{1}}\varsupsetneq Y_{1}R_{P_{1}}$. Now since
$(X_{1}+Y_{1})_{t}=R, $ $X_{1},Y_{1}$ cannot both be in $P_{1}$.
Moreover $X_{1}R_{P_{1}}\varsupsetneq Y_{1}R_{P_{1}}$ ensures that
$X_{1}\nsubseteq P_{1}$,  otherwise $X_{1}$ and $Y_{1}$ would be in
$P_{1}$, and again the strict inclusion implies that $Y_{1}$ must be
in $P_{1}$. From these considerations we have that $X=(X_{1}\Delta
)_{t}$ where $X_{1}\nsubseteq P_{1},$ so $\Delta \subseteq P_{1}$
and because $\Delta \supseteq y^{n}R$ we conclude that $\Delta
\nsubseteq P_{2}\cup P_{3}\cup \dots \cup P_{k}$. As $X=(X_{1}\Delta
)_{t}\subseteq P_{2}\cap\dots\cap P_{k}$ we have $X_{1}\subseteq
P_{2},\dots,P_{k}$.

If $(\Delta +X_{1})_{t}=R$ then, since $X=(X_{1}\Delta )_{t},$ the
set of minimal primes of $X$ is partitioned  into two sets: ones
minimal over $ \Delta $ and ones minimal over $X_{1}$. Since
$X_{1}\subseteq P_{2},\dots ,P_{k}$ we conclude that $X_{1}$ has
$P_{2},\dots ,P_{k}$ for minimal primes and this leaves us with
$\Delta $ with a single minimal prime. Consequently $ \Delta $ is a
packet and $X_{1}$ is a $t$-product of $k-1$ mutually $t$-comaximal
packets, by induction hypothesis. So the theorem is proved in this
case.

Next, let $U=(\Delta +X_{1})_{t}\neq R$ with $\Delta
=(\Delta_{0}U)_{t}$ and $X_{1}=(X_{0}U)_{t}$ and note that
$(X_{0}+\Delta _{0})_{t}=R$. Since $ X_{1}\nsubseteq P_{1},$
$U\nsubseteq P_{1}$ and since $\Delta $ is not contained in any of
$P_{2},P_{3},\dots,P_{k}$ because $\Delta \supseteq y^{n}R$ we
conclude that $U\nsubseteq P_{2},P_{3},\dots,P_{k}$. Consequently $
U\nsubseteq P_{1},P_{2},P_{3},\dots,P_{k}$. But since
$X_{1}=(X_{0}U)_{t}$ is contained in each of
$P_{2},P_{3},\dots,P_{k}$ we conclude that $ X_{0}\subseteq
P_{2}\cap P_{3}\cap \dots \cap P_{k}$. Similarly $\Delta
_{0}\subseteq P_{1}$. We note here that since $P_{1}$ is a minimal prime of 
$X$ and since $X\subseteq \Delta _{0}\subseteq P_{1},$ $P_{1}$ must
be a minimal prime of $\Delta _{0}$. Now we establish that $P_{1}$
is indeed the unique minimal prime of $\Delta _{0}$. For this
assume that there is a prime ideal $Q$ that contains $\Delta
_{0}$. Then as $Q\supseteq \Delta _{0}\supseteq X,$ $Q$ must
contain one of $P_{1},P_{2},P_{3},\dots,P_{k}$. But since
$(X_{0}+\Delta _{0})_{t}=R$ and $X_{0}\subseteq P_{2}\cap
P_{3}\cap
\dots \cap P_{k}$ we conclude that $\Delta _{0}$ is not contained in any of $
P_{2},P_{3},\dots,P_{k}$. Thus $Q$ $\supseteq P_{1}$. Thus $P_{1}$ is the unique minimal
prime of $\Delta _{0}$ and $\Delta _{0}$ is a packet. Thus we have $
X=( X_{1}\Delta)_{t}=(X_{0}U\Delta _{0}U)_{t}=(X_{0}U^{2}\Delta
_{0})_{t}$. Now $(U^{2})_{t}\supseteq X$ $=(X_{0}U^{2}\Delta
_{0})_{t}$ and $U^{2}$ is not contained in any minimal prime of $X$,
so by Lemma \ref{lemma3}, $X\subseteq (U^{2n})_{t}$ for each natural
$n$. As $U$ is $t$-invertible $(X_{0}\Delta _{0})_{t}\subseteq
(U^{2})_{t}$. By Lemma \ref{lemma6}, $(U^{2})_{t}=(U_{1}U_{2})_{t}$
where $X_{0}\subseteq (U_{1})_{t}$ and $\Delta _{0}\subseteq
(U_{2})_{t}$. As $(X_{0}+\Delta _{0})_{t}=R$ we conclude that
$(U_{1}+U_{2})_{t}=R$.  Indeed $
(U_{1}+U_{2})_{t}=(U_{1}+X_{0}+\Delta
_{0}+U_{2})_{t}=(U_{1}+(X_{0},\Delta
_{0})_{t}+U_{2})_{t}=R$. Thus $X=(X_{0}U_{1}U_{2}\Delta _{0})_{t}$ where $%
(X_{0}U_{1}+U_{2}\Delta _{0})_{t}=R$, $P_{1}$ contains, and is
minimal over, $(U_{2}\Delta _{0})_{t}$ and $P_{2},\dots,P_{k}$ are
minimal over $ (X_{0}U_{1})_{t}$. It is easy to see that
$(U_{2}\Delta _{0})_{t}$ is a
packet and that by the induction hypothesis $(X_{0}U_{1})_{t}$ is a $t$%
-product of $k-1$ mutually $t$-comaximal packets. This establishes
the Theorem.

\medskip
 We are now in a position to characterize PVMD URD's.

\begin{thm}\label {Theorem C} A PVMD $R$ is a URD if and only if every nonzero
principal ideal of $R$ has at most finitely many minimal primes.
\end{thm}
\begin{proof} If $R$ is a URD then obviously every proper principal
ideal of $R$ has finitely many minimal primes. Conversely, suppose
that $R$ is a PVMD such that every proper principal ideal of $R$
has finitely many minimal primes and let, by way of a
contradiction, $I$ be a $t$-invertible $t$-ideal
such that $I$ has infinitely many minimal primes. Since $I$ is a $t$-invertible $t$-ideal there exist $x_{1},x_{2},\dots,x_{n}\in I$ such that $I=(x_{1},x_{2},\dots,x_{n})_{t}$. By Lemma \ref{lemma1}, each
minimal prime of $I$ is a minimal prime of $x_{1}R$ or a minimal prime of $(x_{2},\dots,x_{n})_{t}$. Since $x_{1}R$ has only finitely many minimal primes we conclude that $(x_{2},\dots,x_{n})_{t}$ has infinitely many minimal primes. Using
Lemma \ref{lemma1} and the above procedure we can eliminate in turn
$x_{2},\dots,x_{n-2}$ and this leaves us with the conclusion that
$(x_{n-1},x_{n})_{t}$ has infinitely many minimal primes. But then
$x_{n-1}R$ has infinitely many minimal primes or $x_{n}R$ does and
this, in either case, contradicts the assumption that every proper
principal ideal of $R$ has finitely many minimal primes. We conclude by applying Theorem \ref{Theorem B}.
\end{proof}

PVMD's of $t$-finite character were studied by Griffin
\cite{GrKtype} under the name of rings of Krull type. Since the
$t$-prime ideals
 of a PVMD form a tree,
in a ring of Krull type every nonzero principal ideal has at most finitely
many minimal primes.
Thus we have the following corollary.

\begin{cor}\label{corollary7} A ring of Krull type is a URD.
\end{cor}

Indeed, as a GCD-domain is a PVMD such that every $t$-invertible
$t$-ideal is principal, we recover the characterization of GCD URD's given in \cite{ZURD}.

\begin{cor}\label{corollary9} A GCD domain $R$ is a URD if and only if every
proper principal ideal of $R$ is a finite product of packets.
\end{cor}

Using \cite{ZURD} we can also see how some other generalizations of
UFD's can be considered as special cases of PVMD URD's.

\section{General Approach}

The notion of unique representation of a $t$-invertible $t$-ideal
$X\subneq R$ as a $t$-product of packets developed in Section
\ref{one} is somewhat limited in that it deals with $t$-invertible
$t$-ideals which have a unique minimal prime ideal. To bring more
integral domains under the umbrella of unique representation one may
need to concentrate on the properties of packets. We have the
following thoughts on this.

(1). We can replace in the definition of a packet ``$t$-invertible $t$
-ideal" with ``$\ast $-invertible $\ast $-ideal", for a star operation
$\ast$ of finite type, thus getting the notion of $\ast $-packet (a $\ast
$-invertible $\ast$-ideal with a unique minimal prime ideal), as we
have already defined in the introduction. Since, for 
$\ast$ of finite type, every $\ast$-invertible $\ast$-ideal is  $t$-invertible, 
when we define a $\ast $-URD in the usual fashion as a
domain in which every $\ast$-invertible $\ast$-ideal $X\subneq R$
is expressible uniquely as a $\ast $-product of finitely many
mutually $\ast $-comaximal $\ast$-packets, we get a URD albeit of a
special kind. We can settle for the fact that a $\ast$-URD has two special cases: a URD
and a $d$-URD.  In a $d$-URD every invertible ideal is a product of finitely many (invertible) packets and we may note that our results will show that a URD is not necessarily a $d$-URD. It would be interesting to find a star operation $\ast$ 
of finite type, different from $d$ and $t$, and a URD that is a $\ast$-URD.  

 (2). The other choice depends upon another
property of packets: a packet $X$ is such that if
$X=(X_{1}X_{2})_{t}$ where $X_{1},X_{2}$ are proper $t$-ideals then
$(X_{1}+X_{2})_{t}\neq R$. More generally, we can take a star
operation $\ast $ of finite type and call an ideal $X\subneq R$
$\ast $-pseudo irreducible if $X$ is a $\star$-ideal such that
$X=(X_{1}X_{2})^{\ast }$
and $(X_{1}+X_{2})^{\ast }=R$ implies that $(X_{1})^{\ast }=R$ or $(X_{2})^{\ast }=R$. 
This will allow us to develop  the theory of
factorization of $\ast $-ideals on the same lines as the Unique Comaximal Factorization domains of
Swan and McAdam \cite{A, SM}. In this paper, however, we shall
pursue the $\ast$-URD's.

\smallskip
Given a star operation $\ast$ of finite type, we start by giving conditions for which a proper $\ast$-ideal $X$ of a domain $R$ (not necessarily a PVMD) can be written as a
$\ast$-product in the form $X=(X_1X_2\ldots X_n)^\ast$, where each
$X_i$ is a $\ast$-ideal with  prime radical and the $X_i$'s are
pairwise $\ast$-comaximal $\ast$-ideals. In the case $\ast=d$, this
question was investigated in \cite{BH}. Our first result shows that
if such a factorization exists, then it is unique (up to order).

\begin{lemma}\label{comax} Let $R$ be an integral domain and $\ast$ a star operation on $R$.
Let $X$ and $\{Y_i\}_{i=1}^{n}$ be  nonzero ideals of $R$ such
that $(X+Y_i)^\ast=R$ for each $i$. Then
$(X+\prod_{i=1}^nY_i)^\ast=R$.
\end{lemma}

\begin{proof}  We give a proof in the case $n=2$; the general case follows by induction.
Assume that $(X+Y_1)^\ast=(X+Y_2)^\ast=R$. We have
$R=(X+Y_1)^\ast=(X+Y_1(X+Y_2)^\ast)^\ast=(X+Y_1(X+Y_2))^\ast=(X+Y_1X+Y_1Y_2)^\ast=
(X+Y_1Y_2)^\ast$, as desired.
\end{proof}

\begin{lemma}\label{zlemma} Let $R$ be an integral domain and $\ast$ a star operation on $R$.
Let $X$ and $Y$ be two nonzero ideals of $R$ such that
$(X+Y)^\ast=R$. If $X^\ast\supseteq (ZY)^\ast$, for some nonzero
ideal $Z$, then $X^\ast\supseteq Z^\ast$.

\end{lemma}
\begin{proof}  If $X^\ast\supseteq (ZY)^\ast$, we have $X^\ast\supseteq
X^\ast+(ZY)^\ast\supseteq (ZX)^\ast+(ZY)^\ast$. Applying the
$\ast$-operation, $X^\ast\supseteq ((ZX)^\ast+(ZY)^\ast)^\ast=
(ZX+ZY)^\ast= (Z(X+Y))^\ast=Z^\ast$.
\end{proof}

\begin{thm}\label{uniqueness} Let $R$ be an integral domain and $\ast$ a star operation of finite type
 on $R$. Let $X\subneq  R$ be a $\ast$-ideal  and suppose that $X$
 has a $\ast$-factorization of the form $X=(X_1X_2\ldots
 X_n)^\ast$, where each $X_i$ is a $\ast$-ideal with prime radical and the $X_i$'s
 are pairwise $\ast$-comaximal $\ast$-ideals. Then, up to order, such a
 $\ast$-factorization is unique.
\end{thm}

\begin{proof}
Let $X=(X_1X_2\ldots X_n)^\ast=(Y_1Y_2\ldots Y_m)^\ast$ be two
$\ast$-factorizations of $X$ as in the hypotheses. Let $P_1, P_2,
\ldots, P_n$ be the prime ideals minimal over $X_1, X_2, \ldots, X_n$,
respectively. Then, up to order, $Y_1, Y_2, \ldots, Y_m$ must have
the
 same associated minimal primes. Hence $m=n$. We can assume that
 $\rad(Y_i)=P_i$ for each $i=1,\ldots, n$. Since the $P_i$'s are $\ast$-ideals,
 then $(X_i+Y_j)^\ast=R$ for $i\not=j$.
 For a fixed $i$, set $Z_i=\prod_{j\not=i}Y_j$.
  We have $(X_i+Z_i)^\ast=R$ (Lemma \ref{comax}) and $X_i\supseteq X=(Y_iZ_i)^\ast$. Hence
 by Lemma \ref{zlemma}, $X_i\supseteq Y_i$.
 Similarly, $X_i\sub Y_i$. Hence $X_i=Y_i$, for each $i$.
\end{proof}

We next investigate  the decomposition of $\ast$-ideals into
$\ast$-comaximal $\ast$-ideals. Recall that $\starspec(R)$ is treed
if any two incomparable $\ast$-primes are $\ast$-comaximal.

\begin{prop}\label{fact0}
Let $R$ be a domain and let $\ast$ be a  star operation of finite type on $R$.
If $\starspec(R)$ is treed, the following conditions are
equivalent for a $\ast$-ideal $I\subneq R$:
\begin{itemize}
\item[(i)] $I$ has finitely many minimal primes;
\item[(ii)] $I = (Q_{1}Q_{2}\dots  Q_{n})^\ast$, where
the $Q_{i}$'s are $\ast$-ideals with prime radical;
\item[(iii)] $I=(Q_{1}Q_{2}\dots  Q_{n})^\ast$, where
the $Q_{i}$'s are pairwise $\ast$-comaximal $\ast$-ideals with prime radical.
\end{itemize}
Under any of these conditions, if $I$ is $\ast$-invertible, the $Q_{i}$'s are $\ast$-packets.
\end{prop}

\begin{proof} (i) $\ra$ (iii).  Let $P_1, P_2,\ldots, P_n$ be the prime ideals minimal over
$I$. Since $\starspec(R)$ is treed,  $(P_i+P_j)^\ast=R$
for $i\neq j$. Thus, by Lemma \ref{comax}, $(P_i+\prod_{j\not=i}
P_j)^\ast=R$ for each $i$. Denote by $A_i$ a finitely generated ideal of $R$
such that $A_i\subseteq\prod_{j\not=i} P_j$ and
$(P_i+A_i)^\ast=R$.
 For each $i$,  consider the ideal transform of $A_i$, $T_i=\bigcup_{n\geq 0}(R:A_i^n)$.
 Since $A_i$ is finitely generated, the set $\mc F_i$ of the  ideals  $J\sub R$ such that
  $A_i^n\sub J$ for some $n$
 is a localizing system of finite type on $R$ (for the definitions, see \cite{FL}). Moreover, we have  $T_i=R_{\mc F_i}$.
 Hence the map $E\mapsto E^{\star_i}=E_{\mc F_i}$
 on the set of nonzero $R$-submodules of $K$ defines a semistar operation of finite type on $R$
 \cite[Propositions 2.4 and 3.2]{FL},
 which induces a star operation of finite type on $T_i$ ($T_i^{\star_i}=T_i$),
 here still denoted by  $\star_i$. For more details on the concept of semistar
 operation, see \cite{OM}.

Set $I_i=I_{\mc F_i}\cap R$. We claim that $\rad (I_i)=P_i$; in fact
$\rad (I_{\mc F_i})=(P_i)_{\mc F_i}$. To see this, first note that
$I_{\mc F_i}\not=T_i$; otherwise, $A_i^n\sub I$ for some $n\ge$1,
which is impossible since $(A_i^n+P_i)^\ast=R$. Let  $\mc P\sub T_i$
be a minimal prime of $I_{\mc F_i}=(IT_i)^{\star_i}$.  Then $\mc P$
is a $\star_i$-ideal and $\mc P\cap R$ contains a minimal prime
$P_j$ of $I$. We have $(IT_i)^{\star_i} \sub (P_jT_i)^{\star_i} =
(P_j)_{\mc F_i} \sub (\mc P\cap R)^{\star_i} \sub \mc P^{\star_i}=
\mc P$. But for all $i\neq j$, $A_i\sub P_j$ and so $(P_j)_{\mc
F_i}=T_i$. Hence $P_i$ is the unique minimal prime of $I$ contained
in $\mc P\cap R$ and it follows that  $\mc P= (P_i)_{\mc F_i}$.
Hence rad$(I_i)=P_i$. Also rad$(I_i^\ast)=P_i$, since $P_i$ is a
$\ast$-ideal.

 We next show that $I=(Q_1Q_2\ldots Q_n)^\ast$, where $Q_i=I_i^\ast$.
 For this, since the ideals $I_1, I_2\dots, I_n$
 are pairwise $\ast$-comaximal (by the $\ast$-comaximality of their radicals), it is
  enough to show that $I=(I_1\cap I_2\cap\dots \cap I_n)^\ast$.
 Let $M\in \starmax(R)$. If $I\sub M$, then $P_i\sub M$ for some $i$. But since $(P_i+A_i)^\ast =R$,
 $A_i\nsub M$ and so $I_{\mc F_i}\sub IR_M$.
 Hence $I\sub I_1\cap I_2\cap\dots \cap I_n \sub
  I_{\mc F_1}\cap I_{\mc F_2}\cap\dots \cap I_{\mc F_n} \sub \bigcap_{M\in
   \starmax(R), I\sub M} IR_M =I$ (since $I$ is a $\ast$-ideal).
 It follows that $I=(I_1\cap I_2\cap\dots \cap I_n)^\ast$.

(iii) $\ra$ (ii) is clear.

(ii) $\ra$ (i).  If $I = (Q_{1}Q_{2}\dots  Q_{n})^\ast$, a prime minimal
over $I$ must be minimal over one of the  $Q_{i}$'s. Hence $I$ has finitely many minimal
primes.

\smallskip
 To finish, it is enough to observe that if $I$ is $\ast$-invertible, the $Q_{i}$'s are also $\ast$-invertible.
 \end{proof}

\begin{prop} \label{treed} Let $R$ be a domain and let $\ast$ be a  star operation of finite type on $R$.
Assume that each nonzero principal ideal $X\subneq R$ can be written
in the form $X=(X_1X_2\ldots X_n)^\ast$, where the
$X_i$'s are pairwise $\ast$-comaximal $\ast$-packets. Then $\starspec(R)$ is treed.
\end{prop}
\begin{proof} We proceed as in the case of
$\ast=d$ established in \cite[Theorem 1]{BH}. Assume that $P$ and
$P'$ are two incomparable $\ast$-primes contained in the same
$\ast$-maximal ideal $M$. Let $x \in P \sm P'$ and $y \in P' \sm P$.
Write $xyR = (X_{1}X_{2}\dots  X_{n})^\ast$, where the $X_{i}'s$ are
pairwise $\ast$-comaximal $\ast$-packets. Hence $M$ contains only
one of the $X_{i}'s$, say $X_{1}$. On the other hand, $xyR \sub P
\cap P' \sub M$; hence $X_{1} \sub P \cap P'$. Since $X_{1}$ has
prime radical, then either $x$ or $y$ belongs to $P \cap P'$, a
contradiction. It follows that $\starspec(R)$ is treed.
\end{proof}

The following theorem was proven for $\ast=d$ in  \cite[Theorem 2]{BH}.

\begin{thm}\label{fact1}
    Let $R$ be a domain and let $\ast$ be a  star operation of finite type on $R$. The following conditions are equivalent:
\begin{itemize}

    \item[(i)] $\starspec(R)$ is treed and
    each proper $\ast$-ideal of $R$ has finitely many minimal primes;

\item[(ii)] Each proper $\ast$-ideal $I$ can be (uniquely) written in the form $I =
(Q_{1}Q_{2}\dots  Q_{n})^\ast$, where the $Q_{i}$'s are pairwise
$\ast$-comaximal $\ast$-ideals with prime radical.

\end{itemize}
Under any of these conditions, $R$ is a $\ast$-URD.
\end{thm}
\begin{proof}  If condition (ii) holds,
$\starspec(R)$ is treed by Proposition \ref{treed}. Then we can apply  Proposition \ref{fact0} and Theorem \ref{uniqueness}.
\end{proof}

If $\starspec(R)$ is treed and $R$ has $\ast$-finite character, then
each $\ast$-ideal $I\subneq R$ has only finitely many minimal
primes. Thus the following corollary is immediate:

\begin{cor}\label{corfact} Let $R$ be a domain and let $\ast$ be a  star operation of finite type on $R$
such that $\starspec(R)$ is treed and $R$ has $\ast$-finite
character. Then each $\ast$-ideal $I\subneq R$ can be (uniquely)
written in the form $I = (Q_{1}Q_{2}\dots  Q_{n})^\ast$, where the
$Q_{i}$'s are pairwise $\ast$-comaximal $\ast$-ideals with prime
radical.

In particular $R$ is a $\ast$-URD.
\end{cor}

Since the $t$-Spectrum of a PVMD is treed, we also have the following generalization of  Theorem \ref{Theorem B}.

\begin{thm} \label{PVMD}Let $R$ be a PVMD. The following conditions are
equivalent for a  $t$-ideal $I\subneq R$:
\begin{itemize}

    \item[(i)] $I$ has finitely many minimal
primes;

\item[(ii)] $I =
(Q_{1}Q_{2}\dots  Q_{n})_{t}$, where the $Q_{i}$'s are pairwise
$t$-comaximal $t$-ideals with prime radical.
\end{itemize}
\end{thm}

We now  give a complete characterization of $\ast$-URD's.

\begin{lemma}\label{facsp} Let $R$ be an integral domain and $\ast$ a star operation
 on $R$. Let  $X=(X_1X_2\ldots X_n)^\ast$ and $Y=(Y_1Y_2\ldots Y_m)^\ast$ be
 $\ast$-factorizations of the $\ast$-ideals $X$ and $Y$ into $\ast$-ideals of
 $R$ such that $(X_i+ X_j)^\ast=(Y_i+Y_j)^\ast =R$ for $i\not=j$.  Then
 $(X+Y)^\ast=(\prod_{i, j\ge1}(X_i+ Y_j))^\ast$.
\end{lemma}

\begin{proof}  Let $X$, $Y_1$, $Y_2$ be
ideals of $R$ such that $(Y_1+Y_2)^\ast=R$. Then $(X+Y_1Y_2)^\ast=((X+Y_1)(X+
Y_2))^\ast$. In fact
we have $(X+Y_1)(X+Y_2)=X^2+X(Y_1+Y_2)+Y_1Y_2$. Hence
$((X+Y_1)(X+Y_2))^\ast=((X^2)^\ast+(X(Y_1+Y_2))^\ast+(Y_1Y_2)^\ast)^\ast=
((X^2)^\ast+X^\ast+(Y_1Y_2)^\ast)^\ast=(X+Y_1Y_2)^\ast$.

The proof now follows by double induction.
\end{proof}

\begin{lemma}\label{min} Let $R$ be an integral domain and $\ast$
a star operation of finite type on $R$ such that
$\starspec(R)$ is treed. Let $I=(a_1R+a_2R+\ldots+a_nR)^\ast$
be a $\ast$-finite $\ast$-ideal and let $P$ be a prime ideal such that $P\supseteq
I$. Then $P$ is minimal over $I$ if and only if $P$ is minimal over some
$a_iR$.
\end{lemma}
\begin{proof} Assume that $P$ is a minimal ($\ast$-)prime of $I$. Let
 $P_1, P_2,\ldots, P_n$ be a set of minimal primes of
  $a_1, a_2,\ldots, a_n$, respectively, such that $P_i\sub P$ for
  each $i$. Since $\starspec(R)$ is treed, the set $\{P_i\}_i$ is
  a finite chain. Up to order, we may assume that $P_n$ is the maximal element of this
  chain. Then $I\sub P_n\sub P$, and hence $P=P_n$. Thus $P$ is a
  minimal prime over $a_nR$. The converse is clear.
\end{proof}

\begin{thm}\label{*URD} Let $R$ be an integral domain and $\ast$
a star operation of finite type on $R$. Then the following conditions are
equivalent:
\begin{itemize}
\item[(i)] $R$ is a $\ast$-URD;

\item[(ii)]  Each $\ast$-ideal of finite type $X\subneq R$ can be (uniquely) written
in the form $X=(X_1X_2\ldots X_n)^\ast$, where each
$X_i$ is a $\ast$-ideal with  prime radical and the $X_i$'s
 are pairwise $\ast$-comaximal $\ast$-ideals;

\item[(iii)] Each nonzero principal ideal $X\subneq R$ can be (uniquely) written
in the form $X=(X_1X_2\ldots X_n)^\ast$, where the
$X_i$'s are pairwise $\ast$-comaximal $\ast$-packets;

\item[(iv)] $\starspec(R)$ is treed and each $\ast$-ideal of
finite type $X\subneq R$ has only finitely many minimal primes;

\item[(v)] $\starspec(R)$ is treed and each $\ast$-invertible $\ast$-ideal
$X\subneq R$ has only finitely many minimal primes;

\item[(vi)]  $\starspec(R)$ is treed and each nonzero principal ideal
 $X\subneq R$ has only finitely many minimal primes.
 \end{itemize}
 \end{thm}

 \begin{proof}
The  implications (ii) $\ra$ (i) $\ra$ (iii) and
(iv) $\ra$ (v) $\ra$ (vi) are clear.

(iii) $\ra$ (vi). By Proposition \ref{treed}, $\starspec(R)$ is treed.
Then, by Proposition \ref{fact0}, each nonzero principal ideal
$X$ has finitely many minimal primes.

(vi) $\ra$ (iv)  follows from Lemma \ref{min}.

 (vi) $\ra$ (iii). The factorization follows from Proposition \ref{fact0} and the uniqueness follows from
 Theorem \ref{uniqueness}.

 (iii) $\ra$ (ii). We restrict ourselves to the case where
$X=(aR+bR)^\ast$; the general case follows by induction on the
number of generators of $X$ as a $\ast$-finite $\ast$-ideal. Let
$aR=(A_1A_2\ldots A_n)^\ast$ and $bR=(B_1B_2\ldots B_m)^\ast$ be the
$\ast$-factorizations into $\ast$-ideals as in (iii). By Lemma
\ref{facsp}, $X=(\prod_{i, j\ge1}(A_i+ B_j)^\ast)^\ast$. If
$(A_i+B_j)^\ast\not= R$ for some $i, j$, then $\rad(A_i)\sub
\rad(B_j)$ or $\rad(B_j)\sub \rad(A_i)$ (since, by (iii)$\ra$(vi),
$\starspec(R)$ is treed), and hence
$\rad((A_i+B_j)^\ast)=\rad(A_i)\vee \rad(B_j)$ is a prime ideal. To
conclude, we show that the $\ast$-ideals $X_{ij}=(A_i+B_j)^\ast$ are
pairwise $\ast$-comaximal. Let $i, j, k$ and $l$ such that
$(X_{ij}+X_{kl})^\ast\not=R$. Since $\starspec(R)$ is treed, the
minimal primes over $A_i, B_j, A_k$ and $B_l$, respectively, form a
chain, this forces that $i=k$ and $j=l$.
\end{proof}

In the case $\ast=d$, the equivalence (iii) $\lra$ (vi) of Theorem
\ref{*URD} was proved in  \cite[Theorem 1]{BH}, without uniqueness
consideration. When $\ast=t$, we get a complete characterization of
URD's.

\begin{cor}\label{URD} Let $R$ be an integral domain. Then the following conditions are
equivalent:\begin{itemize}
\item[(i)] $R$ is a URD;
\item[(ii)] Each $t$-ideal of finite type $X\subneq R$ can be (uniquely) written
in the form $X=(X_1X_2\ldots X_n)_t$, where each $X_i$ is
a $t$-ideal with  prime radical and the $X_i$'s
 are pairwise $t$-comaximal $t$-ideals;
\item[(iii)] Each nonzero principal ideal $X\subneq R$ can be (uniquely) written
in the form $X=(X_1X_2\ldots X_n)_t$, where the $X_i$'s
are pairwise $t$-comaximal packets;
\item[(iv)] $\tspec(R)$ is treed and each $t$-ideal of
finite type $X\subneq R$ has only finitely many minimal primes;
\item[(v)] $\tspec(R)$ is treed and each $t$-invertible $t$-ideal $X\subneq R$
has only finitely many minimal primes;
\item[(vi)]  $\tspec(R)$ is treed and each nonzero principal ideal $X\subneq R$
 has only finitely many minimal primes.
\end{itemize}
\end{cor}

\begin{rem}\label{wspectreed} (1). For a star operation $\ast$ of finite type, when the $\ast$-class group  is zero (that is each $\ast$-invertible $\ast$-ideal is principal)  a $\ast$-URD
behaves like a GCD URD, i.e. every proper principal ideal is (uniquely)
expressible as a product of finitely many $t$-comaximal principal ideals with a unique minimal prime. Thus a UFD is a GCD URD. However a UFD is not always a $d$-URD, since its spectrum need not be treed (just take $\mathbb Z[X]$).

(2). By using the characterizations of $\ast$-URD's given in
Theorem \ref{*URD}, we get that, for any domain $R$,  $\tspec(R)$ is treed if and only if  $\wspec(R)$ is treed.

For this,  recall
that $\wmax(R)=\tmax(R)$ and that, if $M$ is a $t$-maximal ideal, all
the primes contained in $M$ are $w$-primes \cite{WM}. Now, assume that $\tspec(R)$ is treed and let $M\in \tmax(R)$. Then $\tspec(R_M)$ is linearly ordered, because any $t$-prime of $R_M$ contracts to a $t$-prime of $R$ contained in $M$. Hence each principal ideal of $R_M$ has a unique minimal ($t$-)prime and so is a $d$-packet. It follows that $R_M$ is trivially a $d$-URD and in
particular, $\Spec(R_M)$ is treed. Hence $\wspec(R)$ is treed.
Conversely, since $t$-ideals are $w$-ideals, if $\wspec(R)$ is treed, also $\tspec(R)$ is treed.
\end{rem}

We now show that a  $\ast$-URD $R$ whose $\ast$-packets are $\ast$-powers of $\ast$-prime
ideals is indeed a Krull domain. For $\ast=d$ this implies that  $R$ is a Dedekind domain \cite[Theorem 9]{BH}.

\begin{thm} Let $R$ be an integral domain and $\ast$ a star operation
of finite type on $R$. Then the following conditions are equivalent:
\begin{itemize}
\item[(i)] $R$ is a $\ast$-URD and each $\ast$-packet is a $\ast$-power of a $\ast$-prime ideal;
\item[(ii)] $R$ is a Krull domain and $\ast=t$.
\end{itemize}
\end{thm}

\begin{proof} (i) $\ra$ (ii). We first show that $R$ is a P$\ast$MD (cf. \cite{FJS}). Let
$P\in\starspec(R)$ and let $0\not= x\in P$ with
$xR=(P_1^{\al_1}P_2^{\al_2}\ldots P_r^{\al_r})^\ast$. Then, there
exists $i$ such that $P_i\sub P$. The ideal $P_i$, being a
$\ast$-invertible $\ast$-prime, is a $\ast$-maximal ideal. Thus
$P=P_i$. Hence each $\ast$-prime is $\ast$-maximal.  It follows that
each $M\in \starmax(R)$ has height one. Also, since $M$ is
$\ast$-invertible, $MR_M$ is a principal ideal. Hence $R_M$ is a
rank-one discrete valuation domain.  Thus $R$ is a P$\ast$MD.

By \cite[Proposition 3.4]{FJS}, $R$ is a
PVMD and $\ast=t$. In particular, $R$ is a URD of $t$-dimension one.
Then $R=\bigcap\{R_M;\,\, M\in\tmax(R)\}$ has $t$-finite character.
Whence $R$ is a Krull domain.

 (ii) $\ra$ (i) is well known.
\end{proof}

The conditions of Theorem \ref{fact1} are not equivalent to $R$ being a $\ast$-URD.
In fact in \cite{BH} the authors give an example of  a Pr\"ufer URD  having an ideal
(not finitely generated) with infinitely many minimal primes. However we now show that
such an example cannot be found among domains $\ast$-independent of $\ast$-finite character.

Given a star operation of finite type $\ast$  on a domain $R$, as in
\cite{AZ}, we say that $R$ is $\ast$-independent  if, for every pair
$M, N$ of distinct $\ast$-maximal ideals, $M\cap N$ does not contain
a nonzero $(\ast$-)prime ideal, equivalently each $\ast$-prime is
contained in a unique $\ast$-maximal ideal. Domains
$\ast$-independent of $\ast$-finite character are called $h$-local
domains for $\ast=d$ and weakly Matlis domains for $\ast=t$.

A domain $R$ is $\ast$-independent of $\ast$-finite character if and
only if, for each nonzero nonunit $x\in R$, the ideal $xR$ is
expressible as a product $(I_{1}I_{2}\dots  I_{n})^\ast$, where the
$I_{j}'s$ are pairwise $\ast$-comaximal $\ast$-ideals (not
necessarily packets) \cite[Proposition 2.7]{AZ}. Thus weakly Matlis
domains do have a
 kind of  factorization as a $t$-product of mutually $t$-comaximal ideals.

\begin{thm} \label{WM0} Let $R$ be an integral domain and let $\ast$ be a  star operation of finite type on $R$.
Assume that  each $\ast$-prime ideal is
contained in a unique $\ast$-maximal ideal. The following conditions
are equivalent:

\begin{itemize}

    \item[(i)] $\starspec(R)$ is treed and $R$ has $\ast$-finite character;

       \item[(ii)] $R$ is a $\ast$-URD;

    \item[(iii)] Each $\ast$-ideal $I\subneq R$ is
uniquely expressible as a product   $I=(Q_{1}Q_{2}\dots
Q_{n})^\ast$, where  the $Q_{i}$'s are pairwise $\ast$-comaximal
$\ast$-ideals with prime radical.

\end{itemize}
\end{thm}
\begin{proof}  (i) $\lra$ (ii) Since each $\ast$-prime ideal is
contained in a unique $\ast$-maximal ideal, the $\ast$-finite
character property is equivalent to
 the condition that each nonzero principal ideal
 has at most finitely many minimal primes. Hence we can apply Theorem \ref{*URD}.

(i) $\ra$ (iii) is Corollary \ref{corfact}.

 (iii) $\ra$ (ii) is clear.
 \end{proof}

 \begin{cor} \label{WM1} Let  $\ast$ be a  star operation of finite type on $R$ and assume that $R$
 is $\ast$-independent of $\ast$-finite character. Then $R$ is a $\ast$-URD if and only if $\starspec(R)$ is treed.

 In particular a weakly Matlis domain $R$ is a URD if and only if $\tspec(R)$ is treed.
 \end{cor}

We shall give an example of a weakly Matlis domain that is not a URD in the next section (Example \ref{WMnotURD}).

If $R$ has $t$-dimension one, then clearly $R$ is $t$-independent and $\tspec(R)$ is treed.
 If $R$ has $t$-dimension one and  $t$-finite character, then $R$ is called a weakly Krull domain.
The following corollary follows also from \cite[Theorem 3.1]{AMZ}.

 \begin{cor} \label{dim1} Let $R$ be a domain of $t$-dimension one.  The following conditions are
equivalent:
\begin{itemize}
    \item[(i)] $R$ is a URD;
\item[(ii)]
 Each $t$-ideal $I\subneq R$ is expressible as a product
$I=(Q_{1}Q_{2}\dots  Q_{n})_t$, where
the $Q_{i}$'s are pairwise $t$-comaximal $t$-ideals with prime radical;
 \item[(iii)] $R$ is a weakly Krull domain.
 \end{itemize}
 \end{cor}

From Corollary \ref{dim1}, we see that examples of  URD's are abundant in the form of
weakly Krull domains as rings of algebraic integers of finite
extensions of the
field of rational numbers or as extensions $K_{1}+XK_{2}[X]$ where $
K_{1}\subseteq K_{2}$ is an extension of fields.

The following proposition is an improvement of Corollary \ref{corollary7}.

\begin{prop} Let $R$ be a ring of Krull type. Then  each proper $t$-ideal $I$ of $R$ can be
 uniquely written in the form $I=
(Q_{1}Q_{2}\dots  Q_{n})_{t}$, where the $Q_{i}$'s are pairwise
$t$-comaximal ideals with prime radical. In particular $R$ is a URD.
\end{prop}
\begin{proof}
Follows from Corollary \ref{corfact}.
\end{proof}

From Theorem \ref{WM0} we also get:

 \begin{prop} \label{WM2} Let $R$ be a PVMD such that each $t$-prime ideal is
contained in a unique $t$-maximal ideal. The following conditions are
equivalent:
\begin{itemize}
\item[(i)] $R$ is a ring of Krull type;
\item[(ii)] $R$ is a weakly Matlis domain;
 \item[(iii)] $R$ is a URD;
     \item[(iv)] For each nonzero nonunit $x\in R$, the ideal $xR$ is
uniquely expressible as a product $(I_{1}I_{2}\dots  I_{n})_{t}$,
where the $I_{j}'s$ are pairwise $t$-comaximal ideals;
\item[(v)] Each $t$-ideal $I\subneq R$ is
uniquely expressible as a product   $(Q_{1}Q_{2}\dots  Q_{n})_{t}$,
where  the $Q_{i}$'s are pairwise $t$-comaximal ideals with prime
radical.
\end{itemize}
\end{prop}

 Recall that a domain $R$ is an almost Krull domain if $R_M$ is a DVR for each $t$-maximal ideal $M$.
 Thus a Krull domain is precisely an almost Krull domain with $t$-finite character.

  \begin{cor} An almost Krull domain is a URD if and only if it is a Krull domain.
 \end{cor}

 We end this section with a discussion of factorization of $t$-ideals in generalized Krull domains.

 We recall that a Pr\"ufer domain $R$ is called strongly discrete if
each nonzero prime ideal $P$ of $R$ is not idempotent and that a
generalized Dedekind domain is a strongly discrete Pr\"ufer domain
such that every nonzero principal ideal of $R$ has at most finitely
many minimal primes \cite{P}. The first author gave a generalization
of these domains introducing the class of generalized Krull domains.
A generalized Krull domain is defined as a PVMD $R$ such that
$(P^2)_t\neq P$, for every prime $t$-ideal $P$ of $R$, i. e., $R$ is
a strongly discrete PVMD, and every nonzero principal ideal of $R$
has at most finitely many minimal primes \cite{ElB1, ElB2}. Here it
may be noted that this generalized Krull domain
 is quite different from the generalized Krull domain of
Ribenboim \cite{R} and a generalized Dedekind domain of Popescu
\cite{P} is quite different from what the third author called a
generalized Dedekind domain in \cite{ZGDD}. Incidentally these
later domains were also studied under the name of pseudo Dedekind
domains by Anderson and Kang in \cite{AK}.

Since a generalized Krull domain is a PVMD such that each $t$-ideal has at most finitely many minimal primes \cite[Theorem 3.9]{ElB1}, 
from Theorem \ref{PVMD}, we immediately get:

\begin{prop} \label{GK} If $R$ is a generalized Krull domain, then each $t$-ideal
 $I\subneq R$ is uniquely expressible as a product
$I=(Q_{1}Q_{2}\dots  Q_{n})_{t}$, where the $Q_{i}$'s are pairwise
$t$-comaximal ideals with prime radical.
 \end{prop}

By the previous proposition, we see that a generalized Krull domain is a URD. As a matter of fact,
taking into account Theorem \ref{Theorem C}, we have the following characterization.

 \begin{prop}\label{corollary8}  A PVMD $R$ is a generalized Krull domain if and only if
$R$ is a URD and $(P^{2})_t\neq P$, for each nonzero prime $t$-ideal $P$ of $R$.
\end{prop}

 For generalized Dedekind domains, Proposition \ref{GK} was proved with different methods
  by the second author of this paper and N. Popescu in \cite[Proposition 2.4]{gp}.
  They also showed that a nonzero ideal $I$ of a  generalized Dedekind domain is divisorial
  if and only if  $I=JP_1\dots P_n$, where $J$ is a fractional invertible ideal and
  $P_1, \dots,  P_n$ ($n\ge 1$) are pairwise comaximal prime ideals \cite[Proposition 3.2]{gp}.
   This result was sharpened by B. Olberding, who proved that  $R$ is
    a generalized Dedekind domain whose ideals are all divisorial, equivalently an $h$-local strongly
     discrete Pr\"ufer domain \cite[Corollary 3.6]{ElBG1}, if and only if each ideal $I$ of $R$ is a product of
      finitely generated and prime ideals \cite[Theorem 2.3]{O2}.
Passing through the $t$-Nagata ring, we now extend Olberding's
result to generalized Krull domains. A first step in this direction
was already done by the first author, who proved that, for each
nonzero ideal $I$ of a generalized Krull domain, $I$ is divisorial
if and only if $I=(JP_1\ldots P_n)_t$, where $J\sub R$ is a finitely
generated ideal and $P_1, \ldots, P_n$ ($n\ge 1$)  are pairwise
$t$-comaximal $t$-prime ideals \cite[Theorem 3.4]{ElB2}.

As in \cite{Kang}, we
    set $N(t)= \{h\in R[X] \mid h \neq 0 \mbox{ and } c(h)_t=R\}$, where $c(h)$ denotes the content
    of $h$, and call the domain $\Na:=R[X]_{N(t)}$ the $t$-Nagata ring
of $R$.   B. G.   Kang proved that $R$ is a P$v$MD if and only if
$\Na$ is a Pr\"ufer (indeed a Bezout) domain \cite[Theorem
3.7]{Kang}. In this case,  the map $I_t\mapsto I\Na$  is a lattice
isomorphism between the lattice of $t$-ideals of $R$ and the lattice
of ideals of $\Na$, whose inverse is the map $J\mapsto J\cap R$, and
$P$ is a $t$-prime (respectively, $t$-maximal) ideal of $R$ if and
only if $PR\langle X \rangle$ is a prime (respectively, maximal)
ideal of $R\langle X \rangle$ \cite[Theorem 3.4]{Kang}.
    We recall that a domain is called $w$-divisorial if each $w$-ideal is divisorial \cite{ElBG1}.
Each $w$-divisorial domain is a weakly Matlis domain \cite[Theorem 1.5]{ElBG1} and
a $w$-divisorial generalized Krull domain is precisely a weakly Matlis domain that is a strongly discrete
PVMD \cite[Theorem 3.5]{ElBG1}. We also recall that, when $R$ is a PVMD, $w=t$ \cite[Theorem 3.1]{Kang}.

    \begin{thm} Let $R$ be a domain.
 The following conditions are equivalent:
\begin{itemize}

\item[(i)] $R$ is a $w$-divisorial generalized Krull domain;

\item[(ii)]
For each nonzero  ideal $I\subneq R$, $I_w=(JP_1\ldots P_n)_w$,
where $J\sub R$ is finitely generated and $P_1, \ldots, P_n$ ($n\ge
1$) are pairwise $t$-comaximal $t$-prime ideals;

\item[(iii)] For each nonzero ideal $I\subneq R$, $I_w=(J_1\dots  J_mP_1\ldots P_n)_w$,
where $J_1,\dots, J_m$ are mutually $t$-comaximal packets and $P_1,
\ldots, P_n$ ($n\ge 1$) are pairwise $t$-comaximal $t$-prime
ideals.
\end{itemize}
\end{thm}
\begin{proof}
(i) $\ra$ (ii) and (iii).  Since $I_w$ is divisorial and $w=t$, (ii)
follows from \cite[Theorem 3.4]{ElB2}. The statement (iii) follows
from (ii) and  Proposition \ref{GK}.

     (iii) $\ra$ (ii) is clear.

(ii) $\ra$ (i).
 Assume that the factorization holds. If $M\in \tmax(R)$, then each nonzero proper ideal
of $R_M$ is of type $IR_M$, for some ideal $I$ of $R$. Let
$I'=IR_M\cap R$, by (ii), $I'_w=(JP_1\ldots P_n)_w$, where $J\sub R$
is finitely generated and $P_1, \ldots, P_n$ ($n\ge 1$ ) are
pairwise $t$-comaximal $t$-prime ideals. Then
$IR_M=I'R_M=(JR_M)(P_1R_M)\ldots (P_nR_M)$. We claim that there is
an $i$ such that $P_iR_M\not=R_M$. Otherwise, $IR_M=JR_M$, and hence
$J\sub I'=JR_M\cap R$, so $J_w\sub I'_w=(JP_1\ldots P_n)_w\sub J_w$.
Hence $J_w=(JP_1\ldots P_n)_w$. Let $N\in \tmax(R)$ such that
$P_1\sub N$, then $JR_N=JP_1R_N$, which is impossible by the
Nakayama lemma. Hence $IR_M$ is a product of a finitely generated
ideal and prime ideals. Thus $R_M$ is a strongly discrete valuation
domain by \cite[Theorem 2.3]{O2}. It follows that $R$ is a strongly
discrete PVMD.

 To show that $R$ is weakly Matlis, it is enough to show that the $t$-Nagata ring $\Na$ is  $h$-local
  \cite[Theorem 2.12]{ElBG2}. But, since $R$ is a PVMD, each nonzero ideal of $\Na$ is of type $I\Na=I_t\Na=I_w\Na$ for
   some ideal $I$ of
$R$ \cite[Theorem 3.14]{Kang}.  Hence each nonzero ideal of $\Na$ is a product of a finitely generated ideal
and prime ideals.
 Again from  \cite[Theorem 2.3]{O2} it follows  that $\Na$ is $h$-local.
\end{proof}

\section{Extensions and Examples of URD's}

 In this section we shall discuss ways of constructing examples of
URD's. In addition to showing that a ring of fractions of a URD is
again a URD we identify the conditions under which a polynomial ring
over a URD is again a URD. Using these two main tools we then delve
into the polynomial ring construction $D^{(S)}=D+XD_{S}[X]$. We show
that if $D$ and $D_{S}[X]$ are URD's then $D^{(S)}$ is a URD if and
only if $\tspec(D^{(S)})$ is treed. This indeed leads to several
interesting results.

Since polynomial ring constructions are our main concern here we
start off with a study of polynomial rings over URD's. We recall
from \cite{HH} and \cite{HZ} that, for each nonzero ideal $I$ of a
domain $D$, we have that $I_{t}[X]=I[X]_{t}$ is a $t$-ideal of
$D[X]$ and that the contraction to $D$ of a $t$-ideal of $D[X]$
which is not an upper to zero is a $t$-ideal of $D$. In addition,
each extended $t$-maximal ideal of $D[X]$ is of type $M[X]$, with
$M\in \tmax(D)$. Note that $f(X)\in M[X]$ if and only if
$c(f)_{t}\sub M$, where $c(f)$ denotes the content of $f(X)$, that
is the ideal of $D$ generated by the coefficients of $f(X)$.

\begin{prop}\label{polynomial}  Let $D$ be a domain and $X$ an indeterminate over $D$.
\begin{itemize}
\item[(1)] If $D[X]$ is a URD, then $D$ is a URD.
\item[(2)] If $D$ is a URD,  then $D[X]$ is a URD if and only if  $\tspec(D[X])$ is treed.
\end{itemize}
\end{prop}

\begin{proof} (1). This follows easily from Corollary \ref{URD}, (i) $\lra$ (vi)
 and the fact that, for each $P\in \tspec(D)$,  $P[X]$ is a $t$-ideal of $D[X]$.

(2). By Corollary \ref{URD}, (i) $\lra$ (vi), we have only to show
that if $D$ is a URD and   $\tspec(D[X])$ is treed, then each
nonzero polynomial $f(X)\in D[X]$ has at most finitely many minimal
primes.

If $f(X)\in D[X]$ is contained in the upper to zero $p(X)K[X]\cap
D[X]$ of $D[X]$, then $p(X)$ divides $f(X)$ in $K[X]$. Thus $f(X)$
is contained in at most finitely many uppers to zero.

Assume that $f(X)$ has a minimal prime $P$ such that $P\cap D\neq
(0)$ and let $N$ be a $t$-maximal ideal of $D[X]$ containing $P$.
Then $N=M[X]$, with $M\in \tmax(D)$ and $c(f)_t\in M$, in particular
$c(f)_t\neq D$.
 Let $Q\sub M$ be a minimal prime of $c(f)_t$.  We have
  $f(X)\in c(f)_t[X]\sub  Q[X]$ and $Q[X]$ is a $t$-ideal of $D[X]$. Then $P\sub Q[X]$
  (since $\tspec(D[X])$ is treed) and  $P$ is the
  only minimal prime of $f(X)$ contained in $Q[X]$. But, since $D$ is a URD,
   $c(f)_t$ has only finitely many minimal primes (Corollary \ref{URD}).
 We conclude that $f(X)$ has finitely many minimal primes.
 \end{proof}

\begin{ex}\label{WMnotURD}
An explicit example of a  URD $D$ such that $D[X]$ is not a URD is
given by $D= \overline{\mb Q}+T\mb R[T]_{(T)}$, where $\overline{\mb
Q}$ is the algebraic closure of the rational field $\mb Q$ in the real field $\mb R$ and $T$
is an indeterminate over $\mb R$. Note that $D$ is quasilocal and
one dimensional with maximal ideal $M= T\mb R[T]_{(T)}$. The
quotient field of $D$ is $K=\mb R(T)$ and the $t$-maximal ideal
$M[X]$ of $D[X]$ contains the uppers to zero $P_u = (X-u)K[X]\cap
D[X]$, for all $u\in{\mb R}\setminus\overline{\mb Q}$. Indeed, let
$u\in {\mb R}\setminus\overline{\mb Q}$. Since $D$ is integrally
closed, by \cite[Corollary 34.9]{Gi}, $P_u=(X-u)(D:D+uD)[X]$.
Clearly, $(D:D+uD)$ is a proper $t$-ideal of $D$ and
$M\sub(D:D+uD)$, and hence $M=(D:D+uD)$. Thus $P_u=(X-u)M[X]\sub
M[X]$, as desired. Now, since uppers to zero are height one
$t$-prime ideals, we conclude that $\tspec(D[X])$ is not treed.

We note that both $D$ and $D[X]$ are weakly Matlis domains. In fact,
if a domain $D$ has a unique $t$-maximal ideal $M$, then $D[X]$ is a
weakly Matlis domain. Indeed, a $t$-prime of $D[X]$ is either a
height one $t$-maximal upper to zero or it is contained in $M[X]$.
In addition, since a nonzero polynomial is contained in at most
finitely many uppers to zero, $D[X]$ has $t$-finite character.
\end{ex}

The previous example can be generalized in the sense of the
following result. We are thankful to David Dobbs for providing
Proposition \ref{dobbs} to the third author. Evan Houston also
helped.

\begin{prop} \label{dobbs}[D. Dobbs] Let $D$ be a quasilocal domain
with maximal ideal $M$ and $X$ an indeterminate over $D$. If the
integral closure of $D$ is not a Pr\"ufer domain, then $M[X]$
contains  infinitely many uppers to zero.
\end{prop}
\begin{proof}  Let $K$ be the quotient field of $D$. By \cite[Theorem]{D},
 there exists an element $u \in K$ such that $D \subset D[u]$ does not satisfy incomparability (INC).
  On the other hand, for
 each positive integer $n$, notice that $D[u^n] \subset D[u]$ is integral and
 hence satisfies INC. As the composite of INC extensions is an INC  extension,
 it follows that $D\subset D[u^n]$ does not satisfy INC. Since $u$  is not integral
 over $D$ (because integral implies INC), $u$ is not a root of  unity, and so the elements
 $u^n$ are all distinct, for $n\geq 1$. Thus, there are infinitely many elements $x\in K$
 (the various  $u^n$) such that $D \subset D[x]$ does not satisfy INC. Equivalently, there
 are infinitely  many $x \in K$ such that $I_x:= \ker(D[X] \to D[x]) \subset M[X]$.
 Each such  $I_x$ is an upper to zero (by the First Isomorphism Theorem). Moreover,
  distinct $x$ always lead to distinct $I_x$. Indeed, if $x$ and $y$ are distinct
   nonzero elements of $K$, then $I_x$ and $I_y$ are distinct, since if $y=a/b$ for
    some nonzero elements $a$ and $b$ of $D$, we have $bX-a$ in $I_y$ but not in $I_x$.
\end{proof}

Dobbs' result allows us to characterize unique representation
polynomial rings.

We recall that a domain $D$ is called a UMT-domain if each upper to
zero $P$ of $D[X]$ is a $t$-maximal ideal \cite{HZ}; this is
equivalent to say that $D_P$ has Pr\"ufer integral closure for each
$P\in \tspec(D)$ \cite[Theorem 1.5]{FGH}. A PVMD is precisely an
integrally closed UMT-domain \cite[Proposition 3.2]{HZ}. The
UMT-property is preserved by polynomial extensions  \cite[Theorem
2.4]{FGH}.

\begin{thm}\label{polynomial2}  Let $D$ be a domain and $X$ an indeterminate over
$D$.  The following conditions are equivalent:
\begin{itemize}

\item[(i)] $D[X]$ is a URD;
\item[(ii)] $D$ is a UMT-domain and a URD.
\end{itemize}
\end{thm}

\begin{proof} (i) $\ra$ (ii). Assume that $D[X]$ is a URD and let $P\in \tspec(D)$.
We claim that $D_P$ has Pr\"ufer integral closure. Deny, by Proposition \ref{dobbs}, the
 ideal $PD_P[X]$ in $D_P[X]$ contains infinitely many uppers to
 zero, say $\{Q_\al\}$. For each $\al$, set $P_\al=Q_\al\cap D[X]$. Then
 the $P_\al$'s are distinct uppers to zero of $D[X]$ contained in $P[X]$,
 which is impossible since $\tspec(D[X])$ is treed. It follows that
 $D$ is a UMT-domain.
 The second assertion follows from Proposition \ref{polynomial}.

(ii) $\ra$ (i). Assume that $D$ is a URD. By Proposition
\ref{polynomial}, we have only to show that  $\tspec(D[X])$ is
treed. Since $D$ is a UMT-domain, uppers to zero of $D$ are
$t$-maximal ideals of height one. Also, if $Q$ is a prime $t$-ideal
of $D[X]$ with $Q\cap D=q\not=(0)$, then $Q=q[X]$ (cf. proof of
Theorem 3.1 in \cite{HZ}). Since $q$ does not contain any two
incomparable prime $t$-ideals, $Q$ does not contain any two
incomparable prime $t$-ideals.
 \end{proof}

\begin{cor} \label{polyPVMD} Let $D$ be a domain and $X$ an indeterminate over
$D$. $D$ is a PVMD and a URD if and only if $D[X]$ is a PVMD and a URD.
\end{cor}

 To give more examples of URD's, we will consider domains of the form $D^{(S)}=D+XD_{S}[X]$, where $S$
is a multiplicative set of $D$, in particular of the form $D+XK[X]$,
where $K$ is the quotient field of $D$ \cite{CMZ}. It was shown in
\cite[Proposition 7]{ZURD} that if $D$ is a GCD domain and $S$ is a
multiplicative set such that $D^{(S)}$ is a GCD domain, then
$D^{(S)}$ is a URD if and only if $D$ is a URD. It is not clear how
the $D^{(S)}$ construction will fare in the general case; however
the case when $D^{(S)}$ is a PVMD appears to be somewhat
straightforward.

 Let us call a domain $R$ a pre-Krull domain if every nonzero
principal ideal of $R$ has at most finitely many minimal primes. So
$R$ is a URD if and only if $R$ is pre-Krull and $\tspec(R)$ is
treed (Corollary \ref{URD}).

\begin{lemma}\label{quotient} Let $R$ be an integral domain and $S$ a
multiplicative set of $R$.
\begin{itemize}
\item[(1)] If $R$ is pre-Krull then  $R_S$ is pre-Krull.
\item[(2)] If $R$ is a URD then $R_S$ is a URD.
\end{itemize}
\end{lemma}
\begin{proof} (1). A nonzero principal ideal of $R_{S}$ can be written as $xR_{S}$, where
$x\in R\backslash \{0\}$ and $xR\cap S=\emptyset$. If $Q$ is a minimal prime of $xR_{S}$,
then $Q\cap R$ is a minimal prime of  $xR$.

(2). By part (1) and Corollary \ref{URD} it is enough to show that
$\tspec(R_S)$ is treed. Note that if $M$ is a prime $t$-ideal in
$R_S$ then $M \cap R$ is a prime $t$-ideal in $R$ \cite[page
436]{Zputting}. Now any two prime $t$-ideals $P$ and $Q$ of $R_S$
contained in $M$ are comparable because $P\cap R$ and $Q\cap R$ are
comparable.
\end{proof}

Next we investigate other examples of URD's using the  $D+XD_{S}[X]$
construction.

\begin{lemma}\label{composite} Let $D$ be a domain, $S$ a multiplicative subset of
$D$ and  $X$ an indeterminate over $D$. Let $I$ be an ideal of
$D^{(S)}$ such that $I\cap S \neq \emptyset$. Then $I=JD^{(S)} = J+
XD_S[X]$ for some ideal $J$ of $D$ with $J\cap S  \neq \emptyset$.
Moreover, $I_t=J_t+ XD_S[X]$.
\end{lemma}
\begin{proof} Since $I\cap S \neq \emptyset$, we have $XD_S[X]\sub I$,
hence $I = J+ XD_S[X]=JD^{(S)}$ for some ideal $J$ of $D$ with
$J\cap S \neq \emptyset$.

For the second statement, we first show
that if $A$ is a finitely generated ideal of $D$ such that $A\cap
S\neq \emptyset$, then $(AD^{(S)})_v=A_vD^{(S)}$.
 By \cite[Lemma 3.1]{Zwell}, $(AD^{(S)})^{-1}=A^{-1}D^{(S)}$, so
 $A_vD^{(S)}(AD^{(S)})^{-1}\sub D^{(S)}$, and hence $A_vD^{(S)}\sub
 (AD^{(S)})_v$. For the reverse inclusion, let $f=a_0+Xg(X)\in
 (AD^{(S)})_v$. Then $f(AD^{(S)})^{-1}\sub D^{(S)}$,
  that is, $fA^{-1}D^{(S)}\sub D^{(S)}$. In particular, $a_0A^{-1}\sub D$.
   So $a_0\in A_v$. Thus $f\in A_v+XD_S[X]=A_vD^{(S)}$. Hence $(AD^{(S)})_v=A_vD^{(S)}$.

   Next, note that if $F$ is a nonzero
finitely generated subideal of $I=J+XD_{S}[X]$ such that $F\cap
S=\emptyset$ then for any
$s\in J\cap S$ we have $F_{v}\subseteq (F,s)_{v}=(A+XD_{S}[X])_{v}$
$=A_{v}+XD_{S}[X]=A_{v}D^{(S)}$, for some finitely generated ideal
$A\sub J$ of $D$. So $I_t=(JD^{(S)})_t=\bigcup\{(AD^{(S)})_v;\,
   A\sub J \,\mbox{finitely generated and}\, A\cap S\neq
   \emptyset\}= \bigcup\{A_vD^{(S)};\,
   A\sub J \,\mbox{finitely generated and}\, A\cap S\neq
   \emptyset\}= J_tD^{(S)}=J_t+XD_S[X]$.
\end{proof}

\begin{thm}\label{prekrull} Let $D$ be a domain, $S$ a multiplicative subset of $D$ and  $X$ an
 indeterminate over $D$. Then
$D^{(S)}=D+XD_{S}[X]$ is a pre-Krull domain if and only if $D$ and
$D_S[X]$ are pre-Krull domains.
\end{thm}

\begin{proof} Suppose that $D$ and $D_S[X]$ are pre-Krull domains. Note that
  each prime ideal of $D^{(S)}$ comes
from either of the following disjoint sets of primes: $\mc L=\{P\in
\Spec(D^{(S)}):P\cap S\neq \emptyset \}$ and $\mc M=\{P\in
\Spec(D^{(S)}):P\cap S=\emptyset \}$. Now let $f(X)\in D^{(S)}$, $f(X)\neq 0$.
Then $ f(X)=a_{0}+Xg(X)$ where $g(X)\in D_{S}[X]$ and $a_{0}\in D$.
We have two cases: (i) $ a_{0}\neq 0$ and (ii) $a_{0}=0$.

(i) Assume that $ a_{0}\neq 0$  and let $P$ be a prime ideal minimal over $f(X)$.
(a) Suppose first that $P\in \mc L$. We claim that $P=p+XD_{S}[X]$,
where $p$ is a prime  minimal over $a_{0}$. This is because $P\cap
S\neq \emptyset $ and so $P=p+XD_{S}[X]$ (Lemma \ref{composite}) and
in addition $Xg(X)\in XD_{S}[X]$, which leaves $a_{0}\in p$. If $p$
is not minimal over $a_{0}D$ and say $ q\subsetneq p$ is, then
$q+XD_{S}[X]$ would  contain $f(X)$ and be properly contained in $P$
contradicting the minimality of $P$. Now as $ a_{0}D $ has finitely
many minimal primes we conclude that a finite number of minimal
primes of $f$ come from $\mc L$. (b) Suppose then that $P\in \mc M$.
Since $P\cap S=\emptyset $ we have $P=PD_{S}[X]\cap D^{(S)}$
\cite[pp. 426-427]{CMZ}. So $P$ is a minimal prime of $f(X)$ if and
only if $PD_{S}[X]$  is a minimal prime of $ f(X)D_{S}[X]$. Since
$D_{S}[X]$ is pre-Krull, $f(X)$ has finitely many minimal primes
from $\mc M$. Combining (a) and (b), we conclude that $f(X)$ has
finitely many minimal primes.

(ii) If $a_{0}=0$ then $f(X)=\frac{X^{r}}{s}g(X)$ where $g(X)\in
D^{(S)}$ and $s\in S$. Now as $f(X)\in XD_{S}[X]$, no prime ideal
that intersects $S$ nontrivially can be minimal over $f(X)$. So the
number of prime ideals minimal over $f(X)$ is the same as the number
of prime ideals minimal over $ fD_{S}[X]$ which is finite because
$D_{S}[X]$ is pre-Krull.

Combining (i) and (ii) we have the conclusion.

For the converse suppose that $D^{(S)}$ is pre-Krull.  Then, as $S$
is a multiplicative set in $D^{(S)}$  and as $D_S[X]=D^{(S)}_S$, we
conclude that $D_S[X]$ is pre-Krull (Lemma \ref{quotient}). Now let
$d$ be a nonzero nonunit of $D$. We claim that $dD$  has finitely
many minimal primes. Suppose on the contrary that $dD$ has
infinitely many minimal primes: for  such a minimal prime $Q$, we
have two cases (i) $Q\cap S\neq\emptyset$ or (ii) $Q\cap
S=\emptyset$.

In case (i) for every minimal prime $Q$ of $dD$, $QD^{(S)}=Q+XD_S[X]$ is a
minimal prime over  $dD^{(S)}$. But as $D^{(S)}$ is pre-Krull, there
can be only finitely many minimal primes of $dD^{(S)}$. Consequently,
there are only finitely many minimal primes $Q$ of $dD$ with $Q\cap
S\neq\emptyset$. In case (ii) $Q\cap S=\emptyset$, then $QD_S$ is a
minimal prime of $dD_S$ and so $QD_S[X]$ is a minimal prime of
$dD_S[X]$. But $dD_S[X]$ has a finite number of minimal primes,
because $D_S[X]$ is pre-Krull, a contradiction.
\end{proof}

\begin{cor}\label{corollaryA} Let $D$ be a domain, $S$ a multiplicative subset of $D$ and  $X$ an
indeterminate over $D$.  If D and
$D_S[X]$ are  URD's,  then $D^{(S)}$ is a URD if and only if
$\tspec(D^{(S)})$ is treed.
\end{cor}
\begin{proof} By Corollary \ref{URD} and Theorem \ref{prekrull}.
\end{proof}

\begin{cor}\label{corollaryB} Let $D$ be a domain with quotient field $K$ and $X$ an indeterminate over $D$. Then
$R=D+XK[X]$ is a URD if and only if $D$ is a URD.
\end{cor}
\begin{proof} This follows Lemma
\ref{composite}, Theorem \ref{prekrull} and  Corollary
\ref{corollaryA}.
\end{proof}

The following is an example of $D$ and $D_{S}[X]$ both URD's without
$D^{(S)}$ being a URD.

\begin{ex}\label{nonURD} Let $V$ be a discrete rank two valuation domain with principal maximal ideal $m$. Set
$p=\bigcap m^{n}$ (such that $V_{p}$ is a DVR). Then $R=V+XV_{p}[X]$
is a non-URD pre-Krull domain with $V$ and $V_{p}[X]$ both URD's.

We show that $\tspec(R)$ is not treed. For this we note that the
maximal ideal $m+XV_{p}[X]$ of $R$ contains two prime ideals $
P_1=XV_{p}[X]$ and $P_2=pV_{p}[X]\cap R$. Both $P_1$ and
$P_2$ are $t$ -ideals because they are contractions of two
$t$-ideals of $V_{p}[X]=R_{V\backslash p}$. Since, in
$V_{p}[X],$ both $pV_{p}[X]$ and $XV_{p}[X]$ are distinct principal
prime ideals and hence incomparable, $P_1$ and $P_2$ are
incomparable.
\end{ex}

\begin{rem} \label{wellbURD} Example \ref{nonURD} shows that if $D$ and $D_{S}[X]$
are URD's,  $D^{(S)}$ is not in general a URD. We next give a
partial answer to when the converse is true.

Recall that
 a $t$-prime ideal $P$ of a domain $R$ is well-behaved
if $PR_P$ is a $t$-prime of $R_P$ and that $R$ is well-behaved if
each $t$-prime of $R$ is well-behaved \cite{Zwell}. Clearly, if $P$
is a well behaved $t$-prime such that $P\cap S = \emptyset$, then
$PD_S$ is a $t$-prime.

Now, suppose that $D^{(S)}$ is a URD. Since URD's are stable under quotients (Lemma \ref{quotient}), $D_{S}[X]=D^{(S)}_S$ is a URD.
We know that $D$ is pre-Krull (Theorem \ref{prekrull}). If moreover $D$ is a well
behaved domain, then $\tspec(D)$ is treed and so $D$ is a URD.
Indeed, let $M\in\tmax(D)$ and let $P, Q$ be two $t$-primes contained in $M$.
If $M\cap S=\emptyset$,
as $D$ is well behaved, then $MD_S$ is a
prime $t$-ideal and so are $PD_S,\, QD_S\sub MD_S$. Since $D_{S}[X]$ is a URD, $D_S$ is also a URD (Proposition \ref{polynomial}). Hence $PD_S,\, QD_S$ are comparable and so are $P$
and $Q$. If $M\cap S\neq\emptyset$ then $M+XD_S[X]$ is a $t$-ideal
by Lemma \ref{composite}. Now $P+XD_S[X]$ and $Q+XD_S[X]$ are prime
ideals contained in $M+XD_S[X]$ and because $\wspec(D^{(S)})$ is
treed (Remark \ref{wspectreed}(2)) we conclude that $P+XD_S[X]$ and
$Q+XD_S[X]$ are comparable. Hence $P$ and $Q$ are comparable. So
$\tspec(D)$  is treed.
\end{rem}

\begin{cor}\label{corollaryC} Let $D$ be a PVMD, $X$ an indeterminate over $D$ and $S$
 a multiplicative subset of $D$ such that $D^{(S)}$ is a PVMD. Then
$D^{(S)}$ is a URD if and only if $D$ is a URD.
\end{cor}
\begin{proof} Suppose that $D^{(S)}$ is a URD. Since a PVMD is a well behaved domain \cite{Zwell}, that D is a URD follows from Remark \ref{wellbURD}. Conversely, if $D$ is a URD,  $D_{S}[X]$ is a
 URD by Corollary \ref{polyPVMD}. Applying Corollary \ref{corollaryA} we have that $D^{(S)}$
a URD.
\end{proof}

\bigskip
We have seen in Proposition \ref{GK} that a generalized Krull domain
$D$ is a URD. Thus if $D$ is a generalized Krull domain and
$D^{(S)}$ is a PVMD, then $D^{(S)}$ is a URD. We now show that if
$D^{(S)}$ is a PVMD then it is actually a generalized Krull domain.

\begin{prop} Let $D$ be a generalized Krull domain and let $S$ be a
multiplicative set in $D$ such that $D^{(S)}$ is a PVMD. Then
$D^{(S)}$ is a generalized Krull domain.
\end{prop}

\begin{proof} Let $P$ be a prime $t$-ideal of $D^{(S)}=D+XD_{S}[X]$ where
$D$ is a generalized Krull domain and $S$ is such that $D^{(S)}$ is
a PVMD. Let us assume that $D^{(S)}$ is not a generalized Krull
domain. Now every principal ideal of $D^{(S)}$ has finitely many
minimal primes (Theorem \ref{prekrull}); in fact both $D$ and $D_S[X]$ are generalized Krull domains \cite[Theorem 4.1]{ElB1} and a generalized Krull domain is pre-Krull \cite[Theorem 3.9]{ElB1}. So $D^{(S)}$ not being a generalized Krull domain
would require that there is a prime $t$-ideal $P$ such that
$(P^{2})_{t}=P$. If $P\cap S\neq \emptyset$ then
$P=p+XD_{S}[X]=pD^{(S)}$ for some $t$-prime ideal $p$ of $D$ with
 $p\cap S\neq \emptyset$ (Lemma \ref{composite}). Now $ P^{2}=p^{2}+XD_{S}[X]$ and
again by Lemma \ref{composite} $(P^{2})_{t}=(p^{2})_{t}+XD_{S}[X]$.
Now $ (p^{2})_{t}+XD_{S}[X]=p+XD_{S}[X]$ forces, by degree
considerations, $ (p^{2})_{t}=p$. But $p$ is a prime $t$-ideal of
$D$ which is assumed to be a generalized Krull domain,
contradiction. Hence $P$ does not intersect $S$ nontrivially. Next,
since $D^{(S)}$ is a PVMD for every nonzero ideal $A$ of $D^{(S)}$,
with $A\cap S=\emptyset$, we have
 $(A(D^{(S)})_{S})_{t}=A_{t}(D^{(S)})_{S}=A_{t}D_{S}[X]$.
Setting $A=P^{2}$, we get $(P^{2}(D^{(S)})_{S})_{t}=(P^{2})_{t}(D^{(S)})_{S}=(P^{2})_{t}D_{S}[X]=PD_{S}[X].
$ As $P^{2}(D^{(S)})_{S}=(P(D^{(S)})_{S})^{2}=(PD_{S}[X])^{2}$ we
have $ ((PD_{S}[X])^{2})_{t}=PD_{S}[X]$. But as $D_{S}[X]$ is a
generalized Krull domain \cite[Theorem 4.1]{ElB1} and $ PD_{S}[X]$ a prime $t$-ideal of
$D_{S}[X]$ this is impossible. Having exhausted all the cases we
conclude that $D^{(S)}$ is a generalized Krull domain.
\end{proof}

Since a Krull  domain is a (Ribenboim) generalized Krull domain, by
Corollary 2.7 of \cite{AAZ} for any multiplicative set $S$ of a
Krull domain $D$ we have that $D^{(S)}$ is a PVMD. Hence we have the following consequence.

\begin{cor}\label{corollary17} If $D$ is a Krull domain and $S$ is any
multiplicative set of $D$, then $D^{(S)}$ is a generalized Krull
domain, and hence a URD.
\end{cor}

\section*{Acknowledgments}

Part of this work was done while the first named author was visiting the Department of Mathematics
of  Universit\`a degli Studi  ``Roma Tre", supported by an INdAM visiting grant.

We are grateful to the referee for making a number of insightful suggestions which helped improve the presentation of the paper a great deal.


\begin{thebibliography}{123}

\bibitem{A} D.D. Anderson, Non-atomic unique factorization in integral
domains, Arithmetical properties of commutative rings and monoids,
1--21, Lect. Notes Pure Appl. Math., 241, Chapman \& Hall/CRC, Boca
Raton, FL, 2005.

\bibitem{AAZ} D.D. Anderson, D.F. Anderson and M. Zafrullah, The ring $D+XD[1/S][X]$ and $t$-splitting sets, Commutative algebra. Arab. J. Sci. Eng. Sect. C Theme Issues 26 no. 1 (2001), 3--16.

\bibitem{AC} D.D. Anderson and S. Cook, Two star-operations and their
induced lattices, Comm. Algebra 28 (2000),  2461--2475.

\bibitem{AK} D.D. Anderson and B. Kang, Pseudo-Dedekind domains and
divisorial ideals in $R[X]_{T}$, J. Algebra 122 (1989),
323--336.

\bibitem{AMZ} D.D. Anderson, J. Mott and M. Zafrullah, Finite character
representations for integral domains, Bollettino U.M.I. (7) 6-B
(1992), 613--630.

\bibitem{AZ} D.D. Anderson and M. Zafrullah, Independent locally finite
intersections of localizations, Houston J. Math. 25(1999) 433-452

\bibitem{BZ} A. Bouvier and M. Zafrullah, On some class groups of an
integral domain, Bull. Soc. Math. Gr\`{e}ce (N.S.) 29 (1988),
45--59

\bibitem{BH} J. W. Brewer and W. J. Heinzer, On decomposing
ideals into products of comaximal ideals, Comm. Algebra
30 (2002), 5999--6010.

\bibitem{C} P. Cohn, Bezout rings and their subrings, Proc. Cambridge
Philos. Soc. 64 (1968), 251--264.

\bibitem{CMZ} D. Costa, J. Mott and M. Zafrullah, The construction
$D+XD_{S}[X]$, J. Algebra 53 (1978), 423--439.

\bibitem{D} D. Dobbs, On INC-extensions and polynomials with unit content, Canad. Math. Bull. 23 (1980), 37--42.

\bibitem{DM} T. Dumitrescu and R. Moldovan, Quasi-Schreier domains, Math.
Rep. (Bucur.) 5(55) (2003), 121--126.

\bibitem{ElB1}  S. El Baghdadi, On a class of Pr\"ufer $v$-multiplication
domains, Comm. Algebra 30 (2002),  3723--3742.

\bibitem{ElB2} S. El Baghdadi, Factorization of divisorial
ideals in a generalized Krull domain, Rings, modules, algebras,
and abelian groups, 149--160, Lecture Notes in Pure and Appl.
Math., 236, Dekker, New York, 2004.

\bibitem{ElBG1} S. El Baghdadi and S. Gabelli, $w$-divisorial
domains, J. Algebra 285 (2005), 335--355.

\bibitem{ElBG2} S. El Baghdadi and S. Gabelli, Ring-theoretic
properties of P$v$MDs, Comm. Algebra, to appear.

\bibitem{FJS} M. Fontana, P. Jara and E. Santos, Pr\"ufer
$\star$-multiplication domains
 and semistar operations, J. Algebra Appl. 2 (2003),
 21--50.

\bibitem{FGH} M. Fontana, S. Gabelli and E. Houston, UMT-domains and domains with Pr\"ufer integral closure, Comm. Algebra 26 (1998), 1017--1039.

\bibitem{FL} M. Fontana and J.A. Huckaba,  Localizing systems and semistar operations,
Non-Noetherian commutative ring theory, 169--197, Math. Appl., 520,
Kluwer Acad. Publ., Dordrecht, 2000.

\bibitem{gp} S. Gabelli and N. Popescu, Invertible and divisorial ideals of generalized Dedekind domains, J. Pure Appl. Algebra 135 (1999), 237--251.

\bibitem{Gi} R. Gilmer, Multiplicative Ideal Theory, Marcel Dekker, 1972.

\bibitem{Gr} M. Griffin, Some results on Pr\"ufer
$v$-multiplication rings, Canad. J. Math. 19 (1967), 710--722.

\bibitem{GrKtype} M. Griffin, Rings of Krull type,
J. Reine Angew. Math. 229 (1968), 1--27.

\bibitem{HH} J. Hedstrom and E. Houston, Some remarks on star operations,
J. Pure Appl. Algebra 18 (1980), 37--44.

\bibitem{HZ}  E. Houston and M. Zafrullah, On $t$-invertibility II, Comm. Algebra 17 (1989), 1955--1969.

\bibitem{Kang} B. Kang, Pr\"ufer $v$-multiplication domains and the ring $R[X]_{N_{v}}$,
 J. Algebra 123 (1989), 151--170.

\bibitem{K} I. Kaplansky, Commutative Rings, Allyn and Bacon, Boston, 1970.

\bibitem{SM} S. McAdam and R. G. Swan, Unique Comaximal factorization, J.
Algebra 276(1) (2004), 180--192.

\bibitem{MZ} J. Mott and M. Zafrullah, On Pr\"{u}fer $v$-multiplication
domains, Manuscripta Math. 35 (1981), 1--26.

\bibitem{OM} A. Okabe and R. Matsuda, Semistar-operations on
integral domains, Math. J. Toyama Univ. 17 (1994), 1--21.

\bibitem{O2} B. Olberding, Factorization into prime and invertible ideals,
J. London Math. Soc.   62 (2000), 336--344.

\bibitem{P} N. Popescu, On a class of Pr\"{u}fer domains, Rev. Roumaine
Math. Pures Appl. 29 (1984), 777--786.

\bibitem{R} P. Ribenboim,  Anneaux normaux r\'eels a
caract\`ere fini, Summa Brasil. Math. 3 (1956), 213--253.

\bibitem{WM} F. Wang and R. McCasland, On $w$-modules over strong Mori domains, Comm. Algebra 25 (1997), 1285--1306.

\bibitem{ZURD} M. Zafrullah, On unique representation domains, J. Natur.
Sci. and Math. 18 (1978), 19--29.

\bibitem{ZGDD} M. Zafrullah, On generalized Dedekind domains,
Mathematika 33 (1986), no. 2(1987), 285--295.

\bibitem{ZPresch} M. Zafrullah, On a property of pre-Schreier domains,
Comm. Algebra 15 (1987), no. 9, 1895--1920

\bibitem{Zwell} M. Zafrullah, Well behaved prime $t$-ideals, J. Pure Appl. Algebra 65 (1990)
199--207.

\bibitem{Zputting} M. Zafrullah, Putting $t$-invertibility to use,
Non-Noetherian commutative ring theory, 429--457, Math. Appl.,
520, Kluwer Acad. Publ., Dordrecht, 2000.



\end{thebibliography}
\end{document}